\documentclass[11pt,a4paper,openany,oneside]{article}
\usepackage[a4paper,left=2.5cm,right=2.5cm, top=3cm, bottom=3cm]{geometry}
\usepackage{times}
\usepackage{authblk}
\usepackage[latin1]{inputenc}
\usepackage[italian,english]{babel}
\usepackage{mathrsfs}
\usepackage{stmaryrd}
\usepackage{amssymb,amsmath,amsthm}
\usepackage{hyperref}
\usepackage{graphicx}
\usepackage{eurosym}
\usepackage{color}
\usepackage{enumerate}
\usepackage{multirow} 
\usepackage{mathtools}
\usepackage{scalerel}
\usepackage{booktabs}
\usepackage{float}
\usepackage{amsfonts}
\usepackage{cases}
\usepackage{caption}
\usepackage{todonotes}
\usepackage{relsize}

\theoremstyle{plain}
\newtheorem{thm}{Theorem}[section]

\newtheorem{lemma}[thm]{Lemma}
\newtheorem{prop}[thm]{Proposition}

\newtheorem{example}{Example}

\theoremstyle{definition}

\theoremstyle{remark}
\newtheorem{remark}{Remark}

\usepackage{tikz}
\usepackage{pgfplots}
\pgfplotsset{compat=1.18}
\usetikzlibrary{patterns,calc,intersections,decorations.pathreplacing,pgfplots.fillbetween}

\definecolor{myred}{rgb}{0.84,0.07,0.14}

\newcommand{\R}{\mathbb R}

\newcommand{\Omegabar}{\overline\Omega}
\newcommand{\dOmega}{{\partial\Omega}}
\newcommand{\dB}{{\partial B}}
\newcommand{\intOmega}{\int_\Omega}

\newcommand{\etal}{\e^{t^\alpha}}
\newcommand{\etb}{\e^{t^\beta}}
\newcommand{\thetabar}{\overline\theta}

\usepackage[normalem]{ulem}

\def\e{{\rm e}}

 \def\dd{\, {\rm d}}

\newcommand{\tu}{\widetilde u}
\newcommand{\tv}{\widetilde v}
\newcommand{\tw}{\widetilde w}

\newcommand{\bv}{\overline v}

\newcommand{\ges}{\gtrsim}
\newcommand{\loglog}{\log\log}

\newcommand{\CR}{\mathfrak C_R}

\DeclareOldFontCommand{\it}{\normalfont\itshape}{\mathit}

\newcommand{\supp}{\text{\rm supp}}

\newcommand{\slfrac}[2]{\raisebox{2pt}{$#1$} \hspace{-2pt} \raisebox{1pt}{\large$/$} \hspace{-2pt} \raisebox{-2pt}{$#2$}}
\newcommand{\tslfrac}[2]{\raisebox{1pt}{$\scriptstyle #1$}\hspace{-1pt}\raisebox{0.5pt}{\large$\scriptstyle/$}\hspace{-1pt}\raisebox{-1pt}{$\scriptstyle #2$}}

\numberwithin{equation}{section}
\begin{document}
	\title{{\Large A blow-up approach for a priori bounds in semilinear planar elliptic systems:\\ the Brezis-Merle critical case}}
\author{Laura Baldelli\thanks{Institute for Analysis, Karlsruhe Institute of Technology (KIT) D-76128 Karlsruhe, Germany. \texttt{laura.baldelli@kit.edu}}, \ 
	Gabriele Mancini\thanks{Dipartimento di Matematica, Universita degli Studi di Bari Aldo Moro, Via Orabona 4, 70125 Bari, Italy. \texttt{gabriele.mancini@uniba.it}}, \ 
	Giulio Romani\thanks{Dipartimento di Scienze Matematiche, Informatiche e Fisiche, Universit\`{a} degli Studi di Udine, Via delle Scienze 206, 33100 Udine, Italy. \texttt{giulio.romani@uniud.it}}}
\date{\today}
\maketitle
	
\begin{abstract}
	We establish uniform a priori estimates for solutions of semilinear planar Hamiltonian elliptic systems in a ball with Dirichlet boundary conditions. We consider a broad class of coupled nonlinearities with asymptotic critical behaviour in the sense of Brezis--Merle. The approach we follow is based on a blow-up analysis combined with Liouville--type theorems and integral estimates. Our results extend the scalar theory of uniform a priori bounds to the Hamiltonian case, and solve an open problem  in \cite{dFdOR}. We believe that this approach is new in this setting. As a consequence of our a priori estimates, we prove the existence of a positive solution by means of fixed-point index theory.
\end{abstract}

\section{Introduction}
Let $\Omega\subset\R^N$ be a smooth bounded domain and consider the semilinear Hamiltonian elliptic system with Dirichlet boundary conditions
\begin{equation}\label{sys}
	\begin{cases}
		-\Delta u=f(v)\quad&\mbox{in }\Omega\,,\\
		-\Delta v=g(u)\quad&\mbox{in }\Omega\,,\\
		u=v=0\quad&\mbox{on }\dOmega\,.
	\end{cases}
\end{equation}
It is well known that for $N\geq3$ polynomial growth conditions at infinity have to be fulfilled in order to set these variational problems in Sobolev spaces and prove existence of solutions. If $N\geq3$, for systems of the kind \eqref{sys}, the good notion of criticality is represented by a critical hyperbola: in the model case $f(v)=|v|^{p-1}v$ and $g(u)=|u|^{q-1}u$ with $p,q>1$, the critical hyperbola defined as
\begin{equation}\label{hyp}
	\frac1{p+1}+\frac1{q+1}=\frac{N-2}N
\end{equation}
divides existence and nonexistence of classical positive solutions, see \cite{HV,HMV,M}. Note that if \eqref{sys} is set in the entire $\R^N$ space, then it is proved that the Sobolev hyperbola \eqref{hyp} plays the same role for $N=3,4$, while it is still an open problem, under the name of Lane-Emden conjecture, for dimensions $N\ge 5$, see \cite{M96,PQS,S}. In contrast, much faster growths are admissible in dimension two: in particular, exponential nonlinearities can be treated, with the Trudinger-Moser inequality taking the role of the Sobolev embedding theorem valid for $N\geq3$. For the system \eqref{sys} we still can find a notion of critical hyperbola in the setting of Sobolev-Lorentz spaces: to give an idea, if $f(t)\sim\e^{|t|^\alpha}$ and $g(t)\sim\e^{|t|^\beta}$ with $0<\alpha,\beta<+\infty$ (here the symbol $\sim$ means having the same asymptotic growth), then the maximal growth is obtained on the ``conformal'' critical hyperbola
\begin{equation}\label{abhyp}
	\frac1\alpha+\frac1\beta=1\,,
\end{equation}
see \cite{Ruf_sys} and the more recent advance in \cite{dORR}.

When one treats equations or systems which are not necessarily variational (e.g. \eqref{sys} when $f$ and $g$ depend on both variables $u$ and $v$), a typical technique to prove the existence of solutions is the topological degree theory. An essential step in the proof is to provide a priori uniform bounds for solutions in $L^\infty$ norm. In fact, from this one infers that a certain fixed-point (or Leray-Schauder) index in a large ball of the functional space is zero; moreover, by imposing some natural assumptions on the nonlinearities at $0$, the index in a neighbourhood of the origin is not zero, and therefore, by additivity of the index, one finds a nontrivial solution.    

This strategy has been extensively employed in various contexts, and the study of $L^\infty$ a priori bounds for solutions has become a topic of independent interest in the literature. For the scalar case, we refer to the classical papers by Gidas and Spruck \cite{GS} and de Figueiredo, Lions, and Nussbaum \cite{dFLN} for semilinear problems involving subcritical polynomial nonlinearities in dimension $N\geq3$. Such results have been extended to the higher-order and quasilinear settings, and for wider classes of subcritical nonlinearities, see e.g. \cite{RW,AC,Ruiz,Zou,BF_eq,DP}.	In broad terms, these results indicate that a priori bounds for strong as well as distributional solutions can be established for (almost) all subcritical nonlinearities, up to the threshold of Sobolev critical growth. Concerning the system \eqref{sys} with $N\geq3$, the situation is comparable, in the sense that the critical hyperbola still represents the threshold for both existence and $L^\infty$ a priori bounds, see among others \cite{CdFM,Zou_sys,QS,PQS,BF}.

The situation considerably changes in the case of the \textit{conformal dimension}, where the critical threshold for the existence $t\mapsto\e^{t^{\frac N{N-1}}}$ is given by the Trudinger-Moser inequality. Here, as the seminal work of Brezis-Merle \cite{BM} shows for the scalar second-order case in dimension $N=2$, the limiting growth up to which one may expect to find a priori bounds within the class of \textit{distributional} solutions of $-\Delta u=f(x,u)$ with Dirichlet boundary conditions is only $t\mapsto\e^t$. Within this setting, in \cite{BM} the authors prove a priori bounds for nonlinearities bounded from below and above by $\e^t$, under the assumption $\intOmega f(x,u)\dd x\leq\Lambda$; see also \cite{ChenLi}. We note that such a condition follows from the analysis near the boundary already obtained in \cite{dFLN}. More in general, the respective quasilinear problem involving $-\Delta_N$ in $\Omega\subset\R^N$ with $N\geq2$, was investigated in \cite{LRU}, where the nonlinearities behave either like $\e^{pt}$ for some $p>0$, or grow less than $\e^{t^\alpha}$ with $\alpha\in(0,1)$. The authors use regularity estimates in Orlicz spaces to cover the latter case, while extend the arguments from \cite{BM} in the former. This diversity of the method causes a gap between the two classes of nonlinearities, and e.g. $f(t)=\tfrac{\e^t}{t+1}$ could not be treated with such techniques. This gap was filled in \cite{Rom_N}, where instead a blow-up approach was used. The same technique has been also applied in \cite{MR} to the respective problem of higher-order, namely when the operator is $(-\Delta)^m$ in $\R^{2m}$, $m\geq2$, in both cases of Dirichlet and Navier boundary conditions. In these works, the key assumption which allows to deal at once with both classes of nonlinearities in \cite{LRU}, is the existence of
\begin{equation}\label{b}
	\lim_{t\to+\infty}\frac{f'}f(t)=b\in[0,+\infty)\,
\end{equation}
where the case $b>0$ corresponds to the critical setting in the sense of Brezis and Merle, while $b=0$ for all subcritical growths. We point out that with a blow-up approach one may further derive uniform bounds for \textit{weak solutions} of Dirichlet problems with nonlinearity in the class $f(x,u)=u^{p-1}\e^{u^p}$ even for $p\in(1,2]$, that is, up to the critical Trudinger-Moser growth, under different integral bounds of energy type, see \cite[Theorem 2.2]{DT} and \cite[Theorem 2.1]{MMT}; see also \cite{RobWei,D} for related results.

When dealing with planar systems, a priori bounds have been investigated in \cite{dFdOR} for systems of the kind \eqref{sys} where $f,g$ are positive nonlinearities (which may also depend on $x$), and then extended also for nonvariational elliptic systems in \cite{dFdOR2}. First, by moving plane techniques the authors find uniform bounds of the kind $\intOmega f(v)\dd x\leq\Lambda$ and $\intOmega g(u)\dd x\leq\Lambda$. Then, in order to apply either Orlicz spaces techniques or the argument inspired by \cite{BM}, again as in \cite{LRU}, a dichotomy in the assumptions was needed: in fact, they impose that $f$ and $g$ grow either less than $\e^{t^\alpha}$ and $\e^{t^\beta}$, respectively, with $\alpha,\beta>0$ and $\alpha+\beta<2$, or both like $\e^t$. In addition to the already recalled gap between the growths $\e^{t^\alpha}$ and $\e^t$, the limiting case
\begin{equation}\label{alpha+beta=2}
	\alpha+\beta=2\qquad\mbox{with}\qquad\alpha\neq1\neq\beta
\end{equation}
is stated as an open problem in \cite[Remark 1.3]{dFdOR}. The restriction $\alpha+\beta<2$ in \cite{dFdOR} originates from the last step of the proof of \cite[Theorem 1.3]{dFdOR}: when $\alpha+\beta=2$, the argument no longer yields a uniform estimate of the $L^\infty$ norm, so the case $\alpha+\beta=2$ cannot be closed with the Orlicz-space technique used there. Our method instead replaces this technique with a blow-up analysis which combines a suitable scaling, a compatibility condition, mass bounds and Liouville type theorems, as we will see.

The main purpose of this article is indeed to progress with the study of a priori bounds for system \eqref{sys} in the conformal setting $N=2$, by covering this case of limiting nonlinearities (which from now on will be called \textit{critical in the sense of Brezis-Merle}), by performing a blow-up analysis in the spirit of \cite{Rom_N,MR}. For the model nonlinearities $\e^{t^\alpha}$ and $\e^{t^\beta}$, this will correspond to prove such bounds under \eqref{alpha+beta=2}.  In the companion paper \cite{BMR_sub} we treat instead the subcritical case. We note indeed that such a case is not fully covered by \cite[Theorem 1.3]{dFdOR}, since there are nonlinearities for which $b=0$ but however grow more than any function of the kind $\e^{pt^\alpha}$ irrespective of the choices of $p\geq1$ and $\alpha\in(0,1)$. For instance, it is sufficient to consider $g(t)=\e^t$ and $f(t)=\e^{\frac t{\log t}}$ so that $\slfrac{f'}f(t)=\frac{\log t-1}{(\log t)^2}\sim\frac1{\log t}\to0$, but $f$ grows faster than any function of the kind $\e^{pt^\alpha}$. Crucial steps in our methods will be first to identify the right assumption on $f$ and $g$, which detects the ``critical'' nonlinearities and extends \eqref{b} for systems, and next to determine a suitable scaling, that allows to find a limit system in $\R^2$ to which one may apply a Liouville-type argument. During the blow-up analysis, several technical difficulties will arise due to the non-scalar nature of the problem. In fact, in the proof there will be a continuous exchange of information between the two components of the system. As a first contribution, we restrict to consider the symmetric case $\Omega=B_1(0)$. In this setting, under our global monotonicity assumptions $(H_2)$, we can take advantage of the radial symmetry and monotonicity of the positive solutions $(u,v)$, which implies that both components attain their maxima at the origin. It simplifies to some extent our arguments. The extension to general smooth bounded planar domains remains open, despite the availability of uniform $L^1$ bounds proved in \cite{dFdOR}.

We also mention that blow-up techniques for elliptic systems have been extensively employed in the framework of Toda-type systems, where, differently from \eqref{sys}, a strong algebraic structure and integrability properties allow a detailed analysis of bubbling solutions and mass quantization phenomena, see e.g.  \cite{JostWang,LinWeiZhao,LinWeiYangZhang} and references therein. However, the techniques developed in that setting strongly exploit the specific Cartan matrix structure of the  system and therefore do not appear to be directly transferable to the Hamiltonian systems studied here. To the best of our knowledge, we believe that the approach we follow here is new in the literature, and we hope it could serve for other kinds of problems.
\vskip0.2truecm
Let us now enter into the details of our results.

\subsection{Assumptions and main results}

Having in mind the model growths
\begin{equation}\label{NL_model}
	f(t)\sim\etal,\quad g(t)\sim\etb\qquad\mbox{with}\ \ \alpha>0,\ \beta>0
\end{equation}
with the critical condition
\begin{equation*}
	\alpha+\beta=2\,,
\end{equation*}
we consider the following assumptions:
\begin{enumerate}
	\item[$(H_1)$] $\ f,g\in C^1(\overline{\R^+})\cap C^2((t_0,+\infty))$ for some $t_0>0$, $f>0$, $g>0$ in $\R^+$, and $\displaystyle\lim_{t\to+\infty}f(t)=\lim_{t\to+\infty}g(t)=+\infty$;
	\item[$(H_2)$] \ $f'(t)>0$ and $g'(t)>0$ for all $t\in\R^+$;
	\item[$(H_3)$] $\ \displaystyle\left(\frac f{f'}\right)'\!\!(t)\to0$ and $\displaystyle\left(\frac{g'}g\right)'\!\!(t)\to0$ as $t\to+\infty$;
	\item[$(H_b)$] \ there exists $\displaystyle\lim_{t\to+\infty}\,{\frac{f'}f(t)\,\frac{g'}g(t)}:=b\in(0,+\infty)$.
\end{enumerate}
Note that $\slfrac{f'}f$ and $\slfrac{g'}g$ are the derivatives of the exponents of $f=\e^{\log f}$ and $g=\e^{\log g}$, respectively. For the model nonlinearities in \eqref{NL_model} the case $b=0$ corresponds to $\alpha+\beta<2$, and therefore will be denoted as \textit{Brezis-Merle subcritical}. On the other hand, $\alpha+\beta=2$ implies $b=\alpha\beta>0$, and we will call it \textit{Brezis-Merle critical case}. In this paper we will only deal with the latter case, postponing the analysis of the subcritical one in the companion paper \cite{BMR_sub}, since substantial differences in the respective blow-up arguments will occur.

In this critical setting, we will also need to distinguish between the case for which the limits of both ratios in $(H_b)$ exist and are finite, from the remaining critical cases. We call the former scenario \textit{Liouville}, which then occurs when there exist $p,q>0$ such that
\begin{equation}\label{BM}
	\lim_{t\to+\infty}\frac{f'}f(t)=p\quad\mbox{and}\quad\lim_{t\to+\infty}\frac{g'}g(t)=q\,.
\end{equation}
In the latter case, which we call \textit{non-Liouville}, we may instead suppose
\begin{equation}\label{nonBM}
	\lim_{t\to+\infty}\frac{f'}f(t)=0\quad\mbox{and}\quad\lim_{t\to+\infty}\frac{g'}g(t)=+\infty
\end{equation}
without loss of generality, since the roles of $f$ and $g$ in \eqref{sys} are symmetric. In this last context, we need some monotonicities and growth conditions on $f$ and $g$, namely we assume that there exists $t_0>0$ such that
\begin{enumerate}
	\item[$(H_4)$] $\ \displaystyle\left(\frac f{f'}\right)'\!\!(t)\geq 0$ and $\displaystyle\left(\frac{g'}g\right)'\!\!(t)\geq 0$ as $t>t_0$;
	\item[$(H_5)$] \ the map $\displaystyle t\mapsto t\,\frac{f'}f(t)$ is increasing and $\displaystyle t\mapsto\frac1t\,\frac{g'}g(t)$ is decreasing in $(t_0,+\infty)$.
\end{enumerate}
Finally, we suppose one of the following conditions on the growth of $f$: there exists $a>1$ such that 
\begin{enumerate}
	\item[$(H_6)$] $\ \displaystyle t\,\frac{f'}f(t)\geq a(\log t)^2$ for all $t>t_0$;
	\item[$(H_6')$] $\ \displaystyle t\,\frac{f'}f(t)\geq a(\log\log F(t))^2$ for all $t>t_0$, where $\displaystyle F(t):=\int_0^tf(\tau)\dd\tau$.
\end{enumerate}

Before presenting the main result of the paper, let us briefly discuss the assumptions introduced above. First, observe that, except $(H_1)$ and $(H_2)$, all assumptions are on the behaviour at $\infty$. While $(H_1)$ and $(H_2)$ are standard conditions, $(H_3)$ ensures subcriticality in the sense of Trudinger-Moser and that the growth of $f$ and $g$ are superpolynomial (see Lemma \ref{TM-subcritiche} below). The key assumption $(H_b)$, which characterises the problem as critical, can be regarded as the vectorial counterpart of \eqref{b}, see $(A3)$ in \cite{Rom_N} or \cite{MR}, where the scalar case is considered. In the non-Liouville setting, i.e., under condition \eqref{nonBM}, assumptions $(H_4)$ and $(H_5)$ provide suitable monotonicity properties, which are instrumental for our arguments. Note that the existence of the two limits in \eqref{nonBM} is a consequence of the assumptions $(H_b)$ and $(H_4)$. Finally, $(H_6)$ and $(H_6')$ are technical assumptions which imply a mild lower bound on the growth of $f$, which has to be of the kind $\e^{\log t(\loglog t)^p}$ (see Lemma \ref{lem_f8_conseq} in Section \ref{Sec_OP}). Note that, although $(H_6)$ is of easier verification than $(H_6')$, the latter permits to include a larger class of nonlinearities. A more detailed analysis of the consequences of our assumptions can be found in Section \ref{Sec_Prel}.

\vskip0.2truecm

Since in the blow-up analysis we will need the radial symmetry of the solutions, as already mentioned in the Introduction, we will only deal with the case where $\Omega$ is a ball in $\R^2$.
Our main result is the following:
\begin{thm}\label{Thm_Uniform_bound}
	Let $\Omega=B_1(0)$ be the unit ball in $\R^2$, assume conditions $(H_1)$, $(H_2)$, $(H_3)$, $(H_b)$, and alternatively
	\begin{enumerate}
		\item[i)] \eqref{BM} holds;
		\item[ii)] \eqref{nonBM}, $(H_4)$, $(H_5)$, and either $(H_6)$ or $(H_6')$ hold.
	\end{enumerate}
	Then there exists a constant $C>0$ such that
	\begin{equation*}
		\|u\|_\infty\leq C\qquad\mbox{and}\qquad\|v\|_\infty\leq C
	\end{equation*}
	for all (possible) solutions $(u,v)$ of system \eqref{sys}.
\end{thm}

Note that, under such growth assumptions, one can easily show that $f$ grows less than or comparable to $\e^{pt}$ for some $p>0$ (see Lemma \ref{Lemma3_4} below). Hence we may apply the regularity result \cite[Theorem 1.1]{dFdOR} and conclude that distributional solutions of \eqref{sys} are indeed classical. Therefore, from now on we will simply talk about \textit{solutions} of \eqref{sys}.
We point out that Theorem \ref{Thm_Uniform_bound} remains valid for radial solutions under a weaker version of $(H_2)$, namely only eventual monotonicity, since the global nature of $(H_2)$ is only required to apply \cite[Theorem 1]{Troy}, see Section \ref{Sec_scaling}.

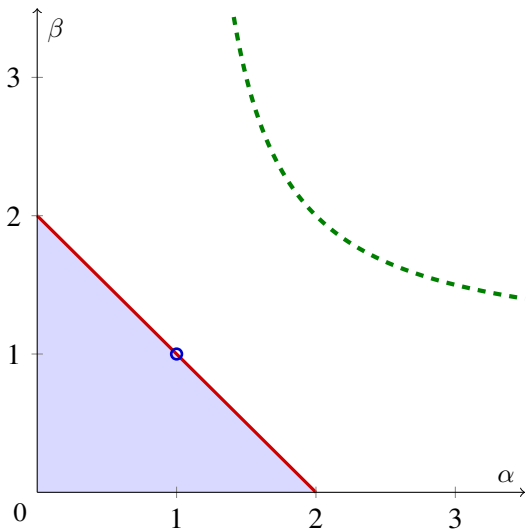
\begin{figure}[h!]
	\centering
	\begin{tikzpicture}
		% Disegno prima il triangolo (sotto gli assi)
		\begin{axis}[
			axis lines=none,
			xmin=0, xmax=3.5,
			ymin=0, ymax=3.5,
			width=8cm, height=8cm,
			clip=false
			]
			\addplot [name path=retta, domain=0:2, samples=2, draw=none] {2-x};
			\addplot [name path=assex, domain=0:2, samples=2, draw=none] {0};
			\addplot [fill=blue!15, draw=none] fill between[of=retta and assex, soft clip={domain=0:2}];
		\end{axis}
		
		% Secondo livello: assi, curve, punto
		\begin{axis}[
			axis lines=middle,
			xlabel={$\alpha$},
			ylabel={$\beta$},
			xmin=0, xmax=3.5,
			ymin=0, ymax=3.5,
			xtick={0,1,2,3},
			ytick={0,1,2,3},
			xticklabels={0,1,2,3},
			yticklabels={0,1,2,3},
			axis line style={->},
			width=8cm, height=8cm,
			clip=false,
			label style={font=\small},
			axis on top
			]
			
			% Retta alpha + beta = 2
			\addplot[red!80!black, very thick, domain=0:2, samples=2] {2 - x};
			
			% Punto (1,1) vuoto al centro con bordo blu scuro
			\addplot[
			only marks,
			mark=o,
			mark size=2pt,
			mark options={draw=blue!80!black, line width=1pt}
			] coordinates {(1,1)};
			
			% Iperbole 1/alpha + 1/beta = 1 (piÃÂ¹ spessa)
			\addplot[green!50!black, dashed, ultra thick, samples=200, domain=1.01:3.5, restrict y to domain=0:3.5]
			{1/(1 - 1/x)};
			
			% Etichetta 0 all'origine
			\node[below left] at (axis cs:0,0) {\small 0};
			
		\end{axis}
	\end{tikzpicture}
	\caption{The state of the art on uniform a priori bounds for \eqref{sys} for the model case \eqref{NL_model}.}
	\label{fig1}
\end{figure}

In Figure \ref{fig1} we can see the state of the art for problem \eqref{sys} for the model nonlinearities in \eqref{NL_model}. In particular, a priori bounds in the blue area and in the point $(1,1)$ were proved in \cite{dFdOR}, the red line is our contribution, while the dashed curve in green is the ``conformal'' critical hyperbola \eqref{abhyp}.
This picture underlines that our case is largely subcritical in the sense of Trudinger-Moser.  Note that we retrieve the scalar results in \cite{LRU,Rom_N} if we restrict on the diagonal of the graph. Therefore, in light of the counterexamples in \cite{BM,MR,Rom_N}, we expect that the region where now the a priori bounds for \eqref{sys} applies is optimal for \textit{distributional} solutions at least in the presence of singular potentials. However, it is still not clear -- even in the scalar case -- whether a priori bounds for \textit{weak} solutions may be proved up to the critical growth in the sense of Trudinger-Moser.

\vskip0.2truecm 
Although reminiscent of some ideas developed in \cite{MR,Rom_N} to deal with corresponding scalar problems, the method that we here present is new and completely different from the one in \cite{dFdOR}, from which we take only the a priori $L^1$ bounds in Proposition \ref{Prop_Lambda}. Our blow-up technique is based on a suitable new scaling for system \eqref{sys}, and the main challenge is to precisely understand how to profit from the exchange of information between the two components $u$ and $v$ in order to find contradictions with such $L^1$-bounds. Let us try to briefly sketch the main line of our argument. We start in Section \ref{Sec_scaling} by assuming by contradiction that there exists a sequence of solutions $(u_k,v_k)_k$ of \eqref{sys} such that 
$$
\max\{\|u_k\|_\infty,\|v_k\|_\infty\}\to+\infty\,.
$$

First, we exclude the possibility that only one among $(\|u_k\|_\infty)_k$ and $(\|v_k\|_\infty)_k$ is unbounded and, exploiting the symmetric setting, we easily see that the origin is the unique blow-up point for both sequences.

The goal is to reach a contradiction either with the uniform $L^1$-bounds on the nonlinear terms (see Proposition \ref{Prop_Lambda}) or by applying a suitable Liouville-type theorem in the plane. In both cases, detecting a suitable scaling is essential: it produces a sequence $(\tu_k,\tv_k)_k$ defined on balls $\Omega_k$ with diverging radii, which formally locally converges to a solution of a Liouville system on $\R^2$. In order to ensure this, we need that the scalings for both components are coupled by a compatibility condition, which we analyse in detail in Section \ref{sec:compatibility}.

Choosing to base the scaling on $\|u_k\|_\infty=u_k(0)$, so that $\tu_k(0)=0$, our blow-up analysis in Section \ref{Sec_blowup} then proceeds by distinguishing several cases according to the behavior of $\tv_k(0)$ as $k\to\infty$. If $\tv_k(0)\to-\infty$ as $k\to\infty$ (Section \ref{Sec_4.1}), then the Harnack inequality can be directly applied to $\tu_k$, which leads to a simple contradiction with the energy bounds. If $\tv_k(0)\to c\in\R$ (Section \ref{Sec_4.2}), then in the non-Liouville case \eqref{nonBM} the boundedness of $\tv_k$ on compact sets yields a similar situation. In the Liouville case \eqref{BM}, we infer the contradiction from the analysis of the limit Liouville system, relying on a characterisation result in \cite{CSW}. Finally, the analysis becomes more technical when $\tv_k(0)\to +\infty$ (Section \ref{+infty}) and we need to find finer global estimates on $\tu_k$ and $\tv_k$, especially in the non-Liouville setting. Here the additional monotonicity assumptions $(H_4)$ and $(H_5)$, together with the growth lower bounds $(H_6)$ and $(H_6')$, will permit to find the desired contradiction.

\vskip0.2truecm

Under a standard superlinear growth condition on the nonlinearities in $0$, in Section \ref{Sec_Existence} we show that once the a priori bound is obtained, then the existence of a positive continuous solution $(u,v)$ of \eqref{sys} follows by means of fixed-point index theory.

\begin{thm}\label{Thm_Existence}
	Besides the conditions under which Theorem \ref{Thm_Uniform_bound} holds, assume moreover
	\begin{enumerate}
		\item[$(H_7)$] $\ \displaystyle\limsup_{t\to0^+}\frac{f(t)}t<\lambda_1\ $ and $\ \,\displaystyle\limsup_{t\to0^+}\frac{g(t)}t<\lambda_1\,$,
	\end{enumerate}
	where $\lambda_1$ is the first eigenvalue of $-\Delta$ in $\Omega$ with Dirichlet boundary conditions. Then \eqref{sys} admits a positive solution.
\end{thm} 

\medskip
\noindent\textbf{Notation.} For $R>0$ and $x_0\in\R^N$ we denote by $B_R(x_0)$ the ball of radius $R$ and center $x_0$. The symbol $o_n(1)$ denotes a vanishing real sequence as $n\to+\infty$. Hereafter, the letter $C$ will be used to denote positive constants which are independent of relevant quantities and whose value may change from line to line.

\section{Preliminary results}\label{Sec_Prel}

In this section, we begin by recalling classical results, as well as some consequences of the assumptions, that will be used throughout the proof. Moreover, we collect some examples of families of nonlinearities, to which our results apply.

\subsection{Elliptic regularity estimates}
We begin by recalling classical local elliptic regularity estimates and a Harnack-type inequality, which will be frequently used in the sequel.
\begin{lemma}[\cite{GT}, Theorems 8.8 and 9.11]\label{LocalEstimates_Serrin}
	Let $u$ be a weak solution of $-\Delta u=h$ in $B_{2R}(0)\subset\R^2$ with $h\in L^2(B_{2R}(0))$. Then,
	\begin{equation*}
		\|u\|_{H^2(B_R(0))}\leq C\big(\|u\|_{L^2(B_{2R}(0))}+\|h\|_{L^2(B_{2R}(0))}\big),
	\end{equation*}
	for a positive constant $C=C(R)$.
\end{lemma}

\begin{lemma}[\cite{Serrin}, Theorem 6]\label{Harnack_Serrin}
	Let $u\geq0$ be a solution of $-\Delta u=h$ in $B_{3R}(0)\subset\Omega\subset\R^2$, where $h\in L^\frac2{2-\varepsilon}(\Omega)$ for some $\varepsilon\in(0,1]\,$. Then
	$$\max_{B_R(0)}u\leq C\big(\min_{B_R(0)}u+K\big),$$
	where the constants $C,K$ depend on $R$ and on $\|h\|_{L^\frac{2}{2-\varepsilon}(\Omega)}$.
\end{lemma}

Crucial for the blow-up analysis will be a uniform estimate on the $L^1$ norm of the nonlinear terms. This in fact holds for general domains since it follows from an application of the moving plane technique and the Kelvin transform.

\begin{prop}[\cite{dFdOR}, Theorem 1.2]\label{Prop_Lambda}
	Let $\Omega$ be a smooth bounded domain in $\R^2$. Assume that {$f,g:\R\to\R^+$ are continuous, non decreasing and} there exist $c>0$ and $p>0$ such that $f(t)\leq c\,\e^{pt}$ for all $t>0$, or that the same holds for $g$. Moreover, let
		\begin{equation}\label{ip:superlin}
			\liminf_{t\to+\infty}\frac{f(t)}{t}\ge a_1>0 \quad \text{and} \quad \liminf_{t\to+\infty}\frac{g(t)}{t}\ge a_2>0\,,\quad \text{with}\,\,\, a_1a_2>\lambda_1^2\,,
		\end{equation}
		where $\lambda_1$ is the first eigenvalue of $(-\Delta,H^1_0(\Omega))$.
	Then there exists a positive constant $\Lambda$, depending only on $f$, $g$, and $\Omega$, such that
	\begin{equation}\label{Lambda}\tag{$\Lambda$}
		\intOmega f(v)\leq\Lambda\quad\mbox{and}\quad\intOmega g(u)\leq\Lambda
	\end{equation}
	for all solutions of \eqref{sys}.
\end{prop}

Note that the assumption needed by Proposition \ref{Prop_Lambda} is satisfied under our hypothesis \eqref{BM} or \eqref{nonBM} since $(H_b)$ implies that $\slfrac{f'}f(t)\leq\nu$ for $t>t_\nu$ for some $\nu>0$ and $t_\nu>0$, which in turn yields $f(t)\leq f(t_\nu)\e^{\nu(t-t_\nu)}$ for $t>t_\nu$. Moreover, assumption ($H_1$) and ($H_3$) imply $f,g$ are superpolynomial at $\infty$, see Lemma \ref{TM-subcritiche} below, hence \eqref{ip:superlin} is satisfied.

\subsection{Consequences of the assumptions}\label{Sec_consequences_ass}

First, we report examples of couples of nonlinearities which verify our assumptions.

\begin{example}
	\begin{enumerate}
		\item[i)] The nonlinearities in \eqref{NL_model}, namely $f(t)=\etal$, $g(t)=\etb$ with $t>t_0>0$, $0<\alpha\le1\le\beta<2$ and the ``critical tuning'' $\alpha+\beta=2$ are our model case. It is easy to verify assumptions $(H_1)$-$(H_6)$ and that $(H_b)$ holds with $b=\alpha\beta$; it is also clear that mild modifications of such $f,g$ in the spirit of \cite{MR,Rom_N}, namely of the kind $(\log t)^\tau t^p\e^{t^\alpha}$ with $0<\alpha<2$, $p\ge0$, and $\tau\in\R$, or $\frac{\e^{\gamma t}}{t^q}$ with $\gamma>0$ and $q\in\R$, also fulfill assumptions $(H_1)$-$(H_6)$ and $(H_b)$.
		\item[ii)] A more general class of growths which is covered by our result is
		\begin{equation}\label{expro}
			f(t)=\e^{t^\alpha(\log t)^p}\quad\mbox{and}\quad g(t)=\e^{\frac{t^\beta}{(\log t)^q}}
		\end{equation}
		for $\alpha,\beta>0$ and $\alpha+\beta=2$, when one chooses $p=q\in\R$. In the case $p>1$, this example also allows somehow to retrieve the borderline case $\alpha=0$, $\beta=2$. Indeed, $f(t)=\e^{(\log t)^p}$ and $g(t)=\e^{\frac{t^2}{(\log t)^q}}$ satisfy $(H_b)$ if one modifies the ``critical tuning'' as $p=q+1$ (in this case $b=2p$). We note that $(H_1)$-$(H_5)$ are verified for all $p>1$. On the other hand, $(H_6)$ is fulfilled in case $p>3$ with $a>1$, and if $p=3$ with $a\in(1,3]$; if, instead, $p\in(1,3)$, then $(H_6)$ does not hold anymore, but $(H_6')$ does.
		\item[iii)] In the ``double-borderline'' case $\alpha=0$ and $p=1$ in \eqref{expro}, it is clear that the nonlinearity $f(t)=\e^{\log t}=t$ is not allowed, but in principle we can keep going modifying it in order to obtain again admissible nonlinearities, e.g. $f(t)=\e^{\log t(\loglog t)^p}$, which may be coupled with $g(t)=\e^{\frac{t^2}{(\loglog t)^p}}$, so that $(H_b)$ holds with $b=2$. Here assumptions $(H_1)$-$(H_5)$ are verified, while $(H_6)$ never holds. However, the more relaxed assumption $(H_6')$ continues to hold, provided $p\geq2$.
	\end{enumerate}
\end{example}

\begin{remark}
	We envision that the lower bound $p\geq2$ in the last example is just a technical obstruction, and that our analysis can be pushed forward, at least up to $p>1$, see the discussion in Section \ref{Sec_OP}.
\end{remark}

From this point on, we present some properties which can be deduced by the conditions we impose on our nonlinearities. The proofs are postponed in the Appendix. In what follows, $(H_1)$ and $(H_2)$ are always assumed.
\begin{lemma}\label{Lemma3_4}
There exist $t_1>0$ such that:
    \begin{enumerate}
        \item[\textit{i})] under \eqref{BM}, there exist $C_1,C_2,\delta_1,\delta_2>0$ such that $f(t)<C_1\e^{\delta_1t}$ and $g(t)>C_2\e^{\delta_2t}$ for $t>t_1$;
        \item[\textit{ii})] under \eqref{nonBM}, one has $f(t)<\e^t<g(t)$ for all $t>t_1$.
    \end{enumerate}
\end{lemma}

\begin{lemma}\label{f3_conseq}
	Under \eqref{BM} or \eqref{nonBM}, condition $(H_3)$ implies 
	\begin{itemize}
		\item[i)] $\displaystyle\left(\frac{f'}{f}\right)'\!(t)\to0\ $ and $\ \displaystyle\left(\frac{g}{g'}\right)'\!(t)\to0$ as $t\to+\infty\,$;
		\item[ii)] $\displaystyle t\,\frac{f'}f(t)\to+\infty\ $ and $\ \displaystyle t\,\frac{g'}g(t)\to+\infty$ as $t\to+\infty\,$;
		\item[iii)] $\displaystyle \frac1t\frac{f'}f(t)\to0\ $ and $\ \displaystyle \frac1t\frac{g'}g(t)\to0$ as $t\to+\infty\,$.
        \item[iv)] $\displaystyle\frac{f'(t)}{f(t)}g(t)\to+\infty$ as $t\to+\infty$.
	\end{itemize}
\end{lemma}

We next observe by $(H_3)$ that both $f$ and $g$ are superpolynomial as well as subcritical in the sense of Trudinger-Moser. Note that the only assumption $(H_b)$ would not have implied either of the two (e.g. $f(t)=t^p$ for $p>1$ and $g(t)=\e^{t^2}$ satisfy $(H_b)$ with $b=2p$).
\begin{lemma}\label{TM-subcritiche}
	Under \eqref{BM} or \eqref{nonBM}, condition $(H_3)$ implies that for all $p>0$ one has
	$$\lim_{t\to+\infty}\frac{f(t)}{t^p}=+\infty=\lim_{t\to+\infty}\frac{g(t)}{t^p}$$
	and, for all $\gamma>0$,
	$$\lim_{t\to+\infty}\frac{f(t)}{\e^{\gamma t^2}}=0=\lim_{t\to+\infty}\frac{g(t)}{\e^{\gamma t^2}}\,.$$
\end{lemma}

\section{Scaling and compatibility condition}\label{Sec_scaling}

Since $\Omega=B_1(0)$ and $(H_1)-(H_2)$ are assumed, by \cite[Theorem 1]{Troy} we know that all solutions of \eqref{sys} are radially symmetric and decreasing along the radial variable. This fact will be essential in our arguments. 

\vskip0.2truecm
We first show that, for a sequence $(u_k,v_k)_k$ of solutions of \eqref{sys}, either there exists a constant $C>0$ such that $\|u_k\|_\infty, \|v_k\|_\infty \leq C$, or 
\begin{equation}\label{unbounded}
	\|u_k\|_\infty=u_k(0)\to+\infty\quad\mbox{and}\quad\|v_k\|_\infty=v_k(0)\to+\infty\,,
\end{equation}
by excluding the possibility that only one between $\|u_k\|_\infty$ and $\|v_k\|_\infty$ is unbounded. Indeed, suppose there is a constant $C>0$ such that $\|u_k\|_\infty\leq C$ and $\|v_k\|_\infty\to+\infty$ as $k\to+\infty$. By $(H_1)$, then $-\Delta v_k=g(u_k)\leq C$, so the comparison principle yields $\max v_k(0)\leq C$, a contradiction. The argument when we switch the roles of $u_k$ and $v_k$, that is $\|u_k\|_\infty\to+\infty$ and $\|v_k\|_\infty\leq C$, is identical.
\vskip0.2truecm
To prove Theorem \ref{Thm_Uniform_bound}, we then need to exclude the case \eqref{unbounded}. We will reason by contradiction, so we will henceforth assume that \eqref{unbounded} holds.

\subsection{The scaling}

A key step in our argument is to identify a nice limit profile of a sequence of rescaled solutions, which formally satisfies a Liouville's equation in the plane. To this aim, crucial is to find a good scaling. Motivated by the usual one in the case of a pure exponential nonlinearity in the scalar case (i.e. $\tu_k(x):=u_k(x_k+\lambda_kx)-M_k$, where $u_k(x_k)=\|u_k\|_\infty=:M_k$ for some suitably defined $\lambda_k\to0$ as $k\to+\infty$, see e.g. \cite{LiSh}), which can also be used to deal with nonlinearities with lower growth as shown in \cite{MR,Rom_N}, we look for an appropriate scaling of the form
\begin{equation}\label{scaling_uk_vk}
	\tu_k(x):=\left(u_k(x_k+\lambda_kx)-S_k\right)A_k\qquad\mbox{and}\qquad\tv_k(x):=\left(v_k(x_k+\lambda_kx)-T_k\right)B_k\,,
\end{equation}
with $A_k,B_k,S_k,T_k,\lambda_k>0$ and $x_k\in\Omega$. Let us then compute the system that $(\tu_k,\tv_k)$ solves. For $x\in\Omega_k:=\lambda_k^{-1}\left(\Omega-x_k\right)$ one has
\begin{equation*}
	\begin{aligned}
		-\Delta\tu_k(x)&=A_k\lambda_k^2f(v_k(x_k+\lambda_kx))=A_k\lambda_k^2f(T_k)\e^{\log(f(v_k(x_k+\lambda_kx)))-\log(f(T_k))}\\
		&=A_k\lambda_k^2f(T_k)\e^{\tslfrac{f'}f(\eta_k(x))(v_k(x_k+\lambda_kx)-T_k)}=A_k\lambda_k^2f(T_k)\e^{\frac1{B_k}\tslfrac{f'}f(\eta_k(x))\tv_k(x)}.
	\end{aligned}
\end{equation*}
Note that in the second-to-last step we used Lagrange's theorem, and linearised $\log f$ around $T_k$ as
\begin{equation*}
	\log f(v_k(x_k+\lambda_kx))=\log f(T_k)+\frac{f'}f(\eta_k(x))(v_k(x_k+\lambda_kx)-T_k)\,,
\end{equation*}
where
\begin{equation}\label{eta_k}
	\eta_k(x):=T_k+\theta_k(x)(v_k(x_k+\lambda_kx)-T_k)=T_k\left(1+\frac{\theta_k(x)}{T_kB_k}\tv_k(x)\right),\qquad\theta_k(x)\in(0,1)\,.
\end{equation}
Analogously for the second equation, one gets
\begin{equation*}
	\begin{aligned}
		-\Delta\tv_k(x)=B_k\lambda_k^2g(S_k)\e^{\frac1{A_k}\tslfrac{g'}g(\zeta_k(x))\tu_k(x)},\qquad x\in\Omega_k\,,
	\end{aligned}
\end{equation*}
where
\begin{equation}\label{zeta_k}
	\zeta_k(x):=S_k+\thetabar_k(x)(u_k(x_k+\lambda_kx)-S_k)=S_k\left(1+\frac{\thetabar_k(x)}{A_kS_k}\tu_k(x)\right),\qquad\thetabar_k(x)\in(0,1)\,.
\end{equation}
We impose that the coefficients in front of the exponential terms are identically $1$, namely
\begin{equation}\label{lambdak}
	\lambda_k^2=\frac1{A_kf(T_k)}=\frac1{B_kg(S_k)}\,.
\end{equation}
We also notice that either $\slfrac{f'}f$ or $\slfrac{g'}g$ appears in the exponents. In order to exploit $(H_b)$, we also define
\begin{equation}\label{B_kA_k}
	A_k:=\frac{f(S_k)}{f'(S_k)}\quad\mbox{and}\quad B_k:=\frac{g(T_k)}{g'(T_k)}\,.
\end{equation}
With these choices one reaches the system
\begin{equation}\label{sys_k}
	\begin{cases}
		-\Delta\tu_k=\e^{\tslfrac{{f'}}f(\eta_k(x))\tslfrac{g'}g(T_k)\tv_k(x)}\quad&\mbox{in }\Omega_k\,,\\
		-\Delta\tv_k=\e^{\tslfrac{{f'}}f(S_k)\tslfrac{g'}g(\zeta_k(x))\tu_k(x)}\quad&\mbox{in }\Omega_k\,,
	\end{cases}
\end{equation}
together with the \textit{compatibility condition} which connects the parameters $S_k$ and $T_k$
\begin{equation}\label{compatibility_condition}
	\frac{{f'(S_k)}}{f(S_k)}g(S_k)=\frac{g'(T_k)}{g(T_k)}f(T_k)\,,
\end{equation}
which follows by combining \eqref{lambdak} and \eqref{B_kA_k}. Note that so far we still have the freedom to choose the points $x_k$ and one of between $S_k$ and $T_k$.
\vskip0.2truecm

\subsection{The compatibility condition}\label{sec:compatibility} In order to better understand the relation \eqref{compatibility_condition} between $S_k$ and $T_k$, let us denote by $J:\R^+\times\R^+\to\R$ the function given by
\begin{equation}\label{J}
	J(s,t)=\frac{f'}f(s)\,g(s)-\frac{g'}g(t)\,f(t)
\end{equation}
with the aim of investigating its zero set. First, note that if $f\equiv g$, then the zero-set coincides with the bisector. In the model case \eqref{NL_model}, this happens when $\alpha=1=\beta$. 

\begin{prop}\label{prop:compatible-scaling}
    Assume $(H_1)$, $(H_2)$, $(H_3)$, and $(H_b)$, together with
    either \eqref{BM} or \eqref{nonBM}.
    There exist $s_0,t_0>0$ and a unique $C^1$ function $t:(s_0,+\infty)\to(t_0,+\infty)$ such that
    \begin{equation}\label{t_s_implicitfct}
		J(s,t(s))=0\qquad\mbox{for all}\ \,s>s_0
	\end{equation}
	and 
	\begin{equation}\label{t_to_infty}
		t(s)\to+\infty\qquad\mbox{as}\ \,s\to+\infty\,.
	\end{equation}
    More specifically, for any $\varepsilon\in(0,1)$, one has 
    \begin{equation}\label{t_F-1G}
	   F^{-1}\left((1-\varepsilon)G(s)\right)<t(s)<F^{-1}\left((1+\varepsilon)G(s)\right), \qquad \text{for $s$ large}.
    \end{equation}
    where $F(t):=\int_0^tf(\tau)\dd\tau$ and $G(t):=\int_0^tg(\tau)\dd\tau$. Under \eqref{BM} with $p<q$ or \eqref{nonBM}, this implies $t(s)>(1-\varepsilon)s$ for $s$ large enough.
\end{prop}

\begin{proof}
We start verifying the conditions of the implicit function theorem. We compute
\begin{equation}\label{J_monotonicity}
	\partial_tJ(s,t)=-\left(\left(\frac{g'}g\right)'\!\!(t)f(t)+\frac{g'}g(t)f'(t)\right)=-f(t)\left(\left(\frac{g'}g\right)'\!\!(t)+\frac{g'}g(t)\frac{f'}f(t)\right)<0
\end{equation}
independently of $s$, for $t$ large by $(H_1)$, $(H_3)$ and $(H_b)$. Then,
\begin{equation}\label{J_t_infty}
	\lim_{t\to+\infty}J(s,t)=-\infty\quad\mbox{for all}\ \,s\ \,\mbox{fixed,}
\end{equation}
by $(H_1)$, and
\begin{equation}\label{J_s_infty}
	\lim_{s\to+\infty}J(s,t)=+\infty\quad\mbox{for all}\ \,t\ \,\mbox{fixed,}
\end{equation}
by Lemma \ref{f3_conseq}(\textit{iv}). Hence there exist $s_0,\,t_0\in\R^+$ such that $\partial_tJ(s,t)<0$ for all $t>t_0$ and $s>0$, and $J(s,t_0)>0$ for all $s>s_0$. Therefore, one may apply the implicit function theorem in $(s_0,+\infty)\times(t_0,+\infty)$ and infer the existence of the implicit function $t=t(s)$ of class $C^1$ such that \eqref{t_s_implicitfct} is satisfied. Moreover, \eqref{t_to_infty} holds. Indeed, if by contradiction for a sequence $s_j\to+\infty$ there holds $t_j:=t(s_j)\leq C$ for some positive constant $C$, then 
$$0=J(s_j,t(s_j))=\frac{{f'(s_j)}}{f(s_j)}g(s_j)-\frac{g'(t_j)}{g(t_j)}f(t_j)\to+\infty$$
by Lemma \ref{f3_conseq}(\textit{iv}) and $(H_1)$, a contradiction.

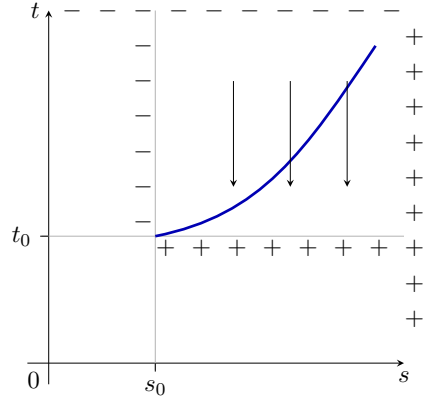
\begin{figure}[h!]\label{Disegnino_implicita}
	\centering
	\scalebox{0.85}{\begin{tikzpicture}[scale=1.1,>=stealth]
			
			% Parametri
			\def\szero{1.5}
			\def\tzero{1.8}
			\def\smax{5}
			\def\tmax{5}
			
			% Assi (con leggera estensione oltre l'origine)
			\draw[->] (-0.3,0) -- (\smax,0);
			\draw[->] (0,-0.3) -- (0,\tmax);
			
			% Etichette assi
			\node[below] at (\smax,0) {$s$};
			\node[left]  at (0,\tmax) {$t$};
			
			% Origine
			\node[below left] at (0,0) {$0$};
			
			% Linea verticale sottile: s = s0
			\draw[gray!60] (\szero,-0.1) -- (\szero,\tmax);
			
			% Linea orizzontale sottile: t = t0
			\draw[gray!60] (-0.1,\tzero) -- (\smax,\tzero);
			
			% Tacche ed etichette
			\draw (\szero,0) -- (\szero,-0.12) node[below] {$s_0$};
			\draw (0,\tzero) -- (-0.12,\tzero) node[left] {$t_0$};
			
			% Curva crescente
			\draw[line width=1.2pt, blue!70!black](\szero,\tzero)
			.. controls (3.0,2.1) and (3.6,3.0)
			.. (4.6,4.5);
			
			% Frecce verticali (stessa altezza e stessa lunghezza)
			\foreach \x in {2.6,3.4,4.2} {
				\draw[->] (\x,4) -- (\x,2.5);
			}
			
			% Segni "+" sul bordo destro (step 0.5)
			\foreach \y in {0.8,1.3,...,4.8} {
				\node at (\smax+0.15,\y-0.17) {$+$};
			}
			
			% Segni "-" in alto (step 0.5)
			\foreach \x in {0.5,1,...,4.5,5} {
				\node at (\x-0.17,\tmax) {$-$};
			}
			
			% Segni "-" a sinistra (step 0.5)
			\foreach \y in {2,2.5,...,4.5} {
				\node at (\szero-0.17,\y) {$-$};
			}
			
			% Segni "+" in basso (step 0.5)
			\foreach \x in {1.5,2,2.5,...,4.5} {
				\node at (\x+0.15,\tzero-0.17) {$+$};
			}
			
	\end{tikzpicture}}
	\caption{A qualitative sketch of the implicit function $t=t(s)$ in blue colour; the arrows indicate the growth of the function $J$ from negative to positive.}
\end{figure}

Moreover, we know that
\begin{equation*}
	\begin{aligned}
		t'(s)&=-\frac{\partial_sJ(s,t(s))}{\partial_tJ(s,t(s))}=\frac{\left(\frac{f'}f\right)'\!\!(s)g(s)+\frac{f'}f(s)g'(s)}{\left(\frac{g'}g\right)'\!\!(t(s))f(t(s))+\frac{g'}g(t(s))f'(t(s))}\\
		&=\frac{g(s)}{f(t(s))}\cdot\frac{\left(\frac{f'}f\right)'\!\!(s)+\left(\frac{f'}f\frac{g'}g\right)(s)}{\left(\frac{g'}g\right)'\!\!(t(s))+\left(\frac{f'}f\frac{g'}g\right)(t(s))}.
	\end{aligned}
\end{equation*}
Note that by \eqref{t_to_infty}, $(H_3)$, Lemma \ref{f3_conseq}(i), and $(H_b)$, the second fraction converges to $1$ as $s\to+\infty$ since $b>0$. Hence,
\begin{equation}\label{ft'g}
	f(t(s))t'(s)=g(s)(1+o(1))\,,
\end{equation}
which, integrated, implies that for any $\varepsilon>0$ there exists $s_1(\varepsilon)>s_0$ so that for all $s>s_1(\varepsilon)$ then
\begin{equation*}
	F(t(s))-F(t(s_1))\begin{cases}
		<(1+\varepsilon)\left(G(s)-G(s_1)\right)\\
		>(1-\varepsilon)\left(G(s)-G(s_1)\right)\,.
	\end{cases}
\end{equation*}

Noticing that $F$ is invertible by $(H_1)$, we deduce an estimate on the behaviour of the implicit function at $\infty$, namely, up to a different $\varepsilon$, that \eqref{t_F-1G} holds. Actually, in case \eqref{nonBM} holds, or \eqref{BM} is satisfied with $p<q$, one can improve this by proving that 
\begin{equation}\label{t_F-1G_new}
	(1-\varepsilon)F^{-1}\left(G(s)\right)<t(s)<(1+\varepsilon)F^{-1}\left(G(s)\right)
\end{equation}
for $s$ large. To this end, define $\gamma(s):=F^{-1}(G(s))$, and we show that
\begin{equation}\label{gamma_s}
	\gamma'(s)=\frac{g(s)}{f(\gamma(s))}\quad\mbox{and}\quad \gamma(s)\geq s\ \ \mbox{for}\ s\ \mbox{large},
\end{equation}
so $\gamma(s)\to+\infty$ as $s\to+\infty$. The first is a direct computation, while the second follows by the fact that, if \eqref{nonBM} holds or \eqref{BM} is satisfied with $p<q$, then $g(s)>f(s)$ for large $s$ by Lemma \ref{Lemma3_4}(\textit{ii}), which clearly implies $G(s)\geq F(s)$. Note, instead, that this is not anymore true when $p=q$, as the choice of $f(t)=\e^{t+\varphi(t)}$ e $g(t)=\e^{t-\varphi(t)}$ with $\varphi(t)=\frac{\sin t}{1+t}$ shows.

Letting $\varepsilon\in(0,1)$ and using $(H_b)$, $(H_3)$, $(H_2)$ and \eqref{gamma_s}, for large $s$ one has
$$\begin{aligned}
	\frac{\dd}{\dd s}J(s,(1+\varepsilon)\gamma(s))&=\partial_sJ(s,(1+\varepsilon)\gamma(s))+(1+\varepsilon)\gamma'(s)\partial_tJ(s, (1+\varepsilon)\gamma(s))\\&=g(s)(b+o(1))-(1+\varepsilon)\gamma'(s)f((1+\varepsilon)\gamma(s))(b+o(1))\\&\le g(s)(b+o(1))-(1+\varepsilon)\gamma'(s)f(\gamma(s))(b+o(1))\\
	&=g(s)(-\varepsilon b+o(1))\to-\infty
\end{aligned}$$
as $s\to+\infty$. This implies that $J(s,(1+\varepsilon)\gamma(s))<0$ which, together with the fact that $t(s)$ is the unique zero of \eqref{J}, then $t(s)<(1+\varepsilon)\gamma(s)$ for $s$ large. Analogously, one infers $t(s)>(1-\varepsilon)\gamma(s)$, and \eqref{t_F-1G_new} is proved. Combining \eqref{t_F-1G_new} and \eqref{gamma_s} yields
\begin{equation*}
	t(s)\geq(1-\varepsilon)s
\end{equation*}
for $s$ large enough.
\end{proof}

Under \eqref{nonBM} we can improve the lower bound on $t(s)$.

\begin{lemma}\label{t>s_Lemma}
	Under \eqref{nonBM}, $(H_1)$-$(H_3)$ and $(H_b)$, if $t=t(s)$ is given by \eqref{t_s_implicitfct}, then 
	$$\displaystyle\lim_{s\to+\infty}\frac{t(s)}s=+\infty\,.$$
	In particular, $t(s)\geq s$ for large $s$. Moreover,
	\begin{equation}\label{t>s_conseq}
		\log F(t(s))\cdot\frac{f'}f(s)\to+\infty\quad\mbox{as}\quad s\to+\infty\,.
	\end{equation}
\end{lemma}
\begin{proof}
	By de l'H\^opital's rule, \eqref{ft'g} in the proof of Proposition \ref{prop:compatible-scaling}, and \eqref{nonBM} one computes
	\begin{equation*}
		\begin{aligned}
			\lim_{s\to+\infty}\frac{t(s)}s&=\lim_{s\to+\infty}t'(s)=\lim_{s\to+\infty}\frac{g(s)}{f(t(s))}=\lim_{s\to+\infty}\frac{g'(s)}{f'(t(s))t'(s)}\\
			&=\lim_{s\to+\infty}\frac{g'}g(s)\cdot\frac f{f'}(t(s))(1+o(1))=+\infty\,.
		\end{aligned}
	\end{equation*}
	
	The second conclusion is immediate from \eqref{t_F-1G} and Lemma \ref{f3_conseq}(\textit{ii}). Indeed, fixing $\varepsilon>0$, for $s$ large, by \eqref{nonBM}, one has
	$$\log F(t(s))\cdot\frac{f'}f(s)\geq\log((1-\varepsilon)G(s))\cdot\frac{f'}f(s)\geq\log(1-\varepsilon)\cdot\frac{f'}f(s)+s\frac{f'}f(s)\to+\infty\,,$$
	having used Lemma \ref{Lemma3_4}(\textit{ii}) to have $\log G(s)>s$ for $s$ large.
\end{proof}

\begin{lemma}\label{log_contradiction}
	In addition to the assumptions of Lemma \ref{t>s_Lemma}, suppose that either $(H_6)$ or $(H_6')$ hold. Then, there exists $\varepsilon>0$ and $s_0>0$ such that
	\begin{equation}\label{log_contradiction_new_step1}
		\frac{t(s)\slfrac{f'}f(t(s))}{\left(\log\slfrac f{f'}(s)\right)^2}>\varepsilon>0
	\end{equation}
	for all $s>s_0$.
\end{lemma}
\begin{proof}
	By simple algebraic manipulations, \eqref{log_contradiction_new_step1} amounts to showing that
	\begin{equation}\label{log_contradiction_new_step2}
		\e^{\sqrt{\frac1\varepsilon t(s)\tslfrac{f'}f(t(s))}}\,\frac{f'}f(s)>1\,.
	\end{equation}
	In case $(H_6)$ is fulfilled, taking $\varepsilon=\frac1a$ and recalling Lemma \ref{t>s_Lemma}, then
	\begin{equation*}
		\e^{\sqrt{\frac1\varepsilon t(s)\tslfrac{f'}f(t(s))}}\,\frac{f'}f(s)>t(s)\frac{f'}f(s)>s\,\frac{f'}f(s)\to+\infty
	\end{equation*}
	by Lemma \ref{f3_conseq}(\textit{ii}), which proves \eqref{log_contradiction_new_step1}. On the other hand, under $(H_6')$, again taking $\varepsilon=\frac1a$,
	\begin{equation*}
		\e^{\sqrt{\frac1\varepsilon t(s)\tslfrac{f'}f(t(s))}}\,\frac{f'}f(s)>\log F(t(s))\cdot\frac{f'}f(s)\to+\infty
	\end{equation*}
	by Lemma \ref{t>s_Lemma}.
\end{proof}

\begin{remark}
	Let us give some examples of behaviours at $\infty$ of the implicit function $t=t(s)$ found above, according to model choices of our nonlinearities.
	\begin{enumerate}
		\item[i)] Let $f(t)=\alpha t^{\alpha-1}\e^{t^\alpha}$ and $g(t)=\beta t^{\beta-1}\e^{t^\beta}$, which satisfy our assumptions, and in particular $(H_b)$ with $b=\alpha\beta$ when $\alpha,\beta\in(0,2)$ and $\alpha+\beta=2$. Then $F(t)=\e^{t^\alpha}$ and $G(t)=\e^{t^\beta}$, therefore $F^{-1}(t)=\left(\log t\right)^\frac1\alpha$. A simple computation gives
		\begin{equation*}
			F^{-1}\left(G(s)\right)=\left(\log\left(\e^{s^\beta}\right)\right)^\frac1\alpha=s^\frac\beta\alpha,
		\end{equation*}
		therefore, \eqref{t_F-1G_new} in the proof of Proposition \ref{prop:compatible-scaling}, yields
		$$(1-\varepsilon)s^\frac\beta\alpha<t(s)<(1+\varepsilon)s^\frac\beta\alpha.$$
		Actually, arguing instead with \eqref{t_F-1G}, we can have a better asymptotic: indeed, by \eqref{t_F-1G} one gets
		$$(s^\beta-\varepsilon)^\frac1\alpha<t(s)<(s^\beta+\varepsilon)^\frac1\alpha.$$
		and
		$$(s^\beta-\varepsilon)^\frac1\alpha=s^{\frac\beta\alpha}\left(1-\frac\varepsilon{s^\beta}\right)^\frac1\alpha=s^\frac\beta\alpha\left(1-\frac1\alpha\,\frac\varepsilon{s^\beta}+o\left(\frac\varepsilon{s^\beta}\right)\right)=s^\frac\beta\alpha-\frac\varepsilon\alpha s^{\frac\beta\alpha(1-\alpha)}+o\left(\varepsilon s^{\frac\beta\alpha(1-\alpha)}\right),$$
		and similarly for the right-hand side, so we also estimate the second term of the asymptotic expansion.
		\item[ii)] If $f,g$ are such that $F(t)=\e^{\left(\log t\right)^p}$ and $G(t)=\e^{\frac{t^2}{(\log t)^q}}$ with $p=q+1$ in order to satisfy $(H_b)$ with $b=2p$, with a similar argument one shows that the leading term of the map $s\mapsto t(s)$ is $\e^{\left(s^2\left(\log s\right)^{-q}\right)^\frac1p}$.
	\end{enumerate}
\end{remark}

\section{The blow-up argument: Proof of Theorem \ref{Thm_Uniform_bound}}\label{Sec_blowup}

In light of the scaling and the compatibility condition investigated in Section \ref{Sec_scaling}, we can now give the proof of Theorem \ref{Thm_Uniform_bound}. Let $(u_k,v_k)_k$ be a sequence of positive solutions of \eqref{sys} and assume by contradiction that the components are unbounded in $L^\infty(B_1(0))$, that is \eqref{unbounded} holds. By the radial symmetry given by $(H_1)-(H_2)$, it is immediate to see that the only blow-up point for $(u_k)_k$ and $(v_k)_k$ is the origin. Indeed, if $\bar{x}\neq0$, $\bar{x}\in B_1(0)$, is such that $\exists\bar{x}_k \to \bar{x}$ with $u_k(\bar{x}_k)\to+\infty$ or $v_k(\bar{x}_k) \to +\infty$, since both $u_k,v_k$ are radially decreasing about the origin, then $u_k\to+\infty$ or $v_k\to+\infty$, resp., uniformly any ball $B_{|\bar r|}(0)$ with $\bar r\in(0,|\bar x|)$. This clearly contradicts the uniform $L^1$ bounds in \eqref{Lambda}, since $f$ and $g$ are increasing by $(H_2)$. 

Let us define the rescaled sequences $(\tu_k)_k$, $(\tv_k)_k$ as in \eqref{scaling_uk_vk}. The natural choice of the scaling parameters is $x_k=0$ and
$$S_k:=\|u_k\|_\infty=u_k(0)\to+\infty\,.$$
Then $T_k=t(S_k)$ where $t=t(s)$ is the implicit function defined by \eqref{t_s_implicitfct} in Proposition \ref{prop:compatible-scaling}, and the parameters $A_k$, $B_k$, and $\lambda_k$ are consequently fixed by \eqref{B_kA_k} and \eqref{lambdak}. Note that 
$$T_k\to+\infty\quad\mbox{and}\quad\lambda_k\to0$$
by \eqref{t_to_infty} and \eqref{lambdak}, respectively, under $(H_1)$ in both settings \eqref{BM} and \eqref{nonBM}, since $f(T_k)\to+\infty$ and $A_k\not\to0$. As a consequence, $\Omega_k$ enlarge towards $\R^2$ as $k\to+\infty$.

By the choice of $S_k$ one immediately has
$$\tu_k\leq0\ \,\mbox{in}\ \,\Omega_k\quad\mbox{and}\quad\tu_k(0)=0\,.$$
Let us also define
\begin{equation}\label{bvk}
	\bv_k(x):=\tv_k(x)-\tv_k(0)\leq0
\end{equation}
since $0$ is the maximum point of $\tv_k$, so $\bv_k(0)=0$ trivially. We have
\begin{equation}\label{eq_bvk}
	-\Delta\bv_k(x)=-\Delta\tv_k(x)=\e^{\tslfrac{f'}f(S_k)\tslfrac{g'}g(\zeta_k(x))\,\tu_k(x)},\quad x\in\Omega_k\,,
\end{equation}

where $\zeta_k$ is defined by \eqref{zeta_k}. Hence by $(H_1)$--$(H_2)$ one infers $0\le-\Delta\bv_k\leq1$ in $\Omega_k$. Therefore, for any fixed $R>0$ and $k$ sufficiently large, one may apply the Harnack inequality in Lemma \ref{Harnack_Serrin} with $\tw_k:=-\bv_k\geq0$ in $\Omega_k$ and get
$$\max_{B_R(0)}|\bv_k|=\max_{B_R(0)}\tw_k\leq C_R\left(\min_{B_R(0)}\tw_k+K(\|\Delta\tw_k\|_\infty,R)\right)\leq C_R\,.$$
By Lemma \ref{LocalEstimates_Serrin}, this implies
\begin{equation}\label{5.1bis}
	\begin{aligned}
		\|\bv_k\|_{H^2(B_{R/2}(0))}&\leq C_R\left(\|\bv_k\|_{L^2(B_R(0))}+\|\Delta\tv_k\|_{L^2(B_R(0))}\right)\\
		&\leq C_R\left(\|\bv_k\|_{L^\infty(B_R(0))}+\|\Delta\bv_k\|_{L^\infty(B_R(0))}\right)\leq C_R\,.
	\end{aligned}
\end{equation}
Hence, by compact embedding and a diagonal argument, there exists $\bv\in H^2_{loc}(\R^2)$ such that 
\begin{equation}\label{bvk_to_bv}
	\bv_k\to\bv\quad\mbox{uniformly on compact sets of $\R^2$}
\end{equation}
(which from now on will be shortened by the acronym \textit{u.c.s.}) and in turn
\begin{equation}\label{tvk_ucs}
	\tv_k(x)=\tv_k(0)+\bv_k(x)=\tv_k(0)+\bv(x)+o_k(1)\quad\mbox{u.c.s.}.
\end{equation}
A comparable local compactness result cannot be obtained for $\tu_k$, since the behaviour of the right-hand side of 
\begin{equation*}
	-\Delta\tu_k(x)=\e^{\tslfrac{f'}f(\eta_k(x))\tslfrac{g'}g(T_k)\,\tv_k(x)},
\end{equation*}
where $\eta_k$ is defined in  \eqref{eta_k}, does depend on the behaviour of $\tv_k$. Hence, in light of \eqref{tvk_ucs}, (possibly passing to a subsequence), we distinguish three cases according to the behaviour of $\tv_k(0)\to\ell\in\{-\infty,\,c\in\R,\,+\infty\}$ as $k\to+\infty$.

\subsection{Case \texorpdfstring{$\boldsymbol{\tv_k(0)\to-\infty}$}{1}}\label{Sec_4.1}
In this setting, by \eqref{tvk_ucs} one also has $\tv_k\to-\infty$ u.c.s., hence $-\Delta\tu_k\in[0,1]$ u.c.s..	Recalling $\tu_k(0)=0$ and $\tu_k\leq0$, as before the Harnack inequality and the local elliptic regularity theory yield the existence of $\tu\in H^2_{loc}(\R^2)$ such that 
\begin{equation}\label{ukti}
	\tu_k\to\tu\qquad\quad\mbox{u.c.s.}.
\end{equation}
Therefore, by Lemma \ref{f3_conseq} (\textit{iii}),
\begin{equation}\label{zeta-Sk}
	\zeta_k(x)=S_k\Big(1+\underbrace{\thetabar_k(x)}_{\in[0,1]}\underbrace{\frac{f'(S_k)}{f(S_k)S_k}}_{\to0}\underbrace{\tu_k}_{\to\tu}\Big)\sim S_k\quad\mbox{u.c.s.}.
\end{equation}
Moreover, we claim that 
\begin{equation}\label{skzk}
	\frac{f'}f(S_k)\frac{g'}g(\zeta_k(x))\sim\frac{f'(S_k)}{f(S_k)}\frac{g'(S_k)}{g(S_k)}\to b\,\qquad \text{u.c.s.}
\end{equation}
since $S_k\to+\infty$. This is clear if \eqref{BM} is assumed, while in the case \eqref{nonBM} one argues as follows. First, we have
\begin{equation}\label{zk-sk}
	\zeta_k(x)-S_k=\underbrace{\thetabar_k(x)}_{\in[0,1]}\underbrace{\frac{f'(S_k)}{f(S_k)}}_{\to0}\underbrace{\tu_k}_{\to\tu}\to0\quad\mbox{u.c.s.},
\end{equation}
that is, $\zeta_k\to+\infty$ u.c.s.. Hence
\begin{equation}\label{proof_asympt_g}
	\frac{g'}g(S_k)-\frac{g'}g(\zeta_k(x))=\int_{\zeta_k(x)}^{S_k}\left(\frac{g'}g\right)'\!(\xi)\dd\xi\leq\sup_{[\zeta_k(x),\,S_k]}\left(\frac{g'}g\right)'\!\cdot\left(S_k-\zeta_k(x)\right)=o_k(1)
\end{equation}
u.c.s., by \eqref{zk-sk}, $(H_3)$, and $(H_4)$. Therefore $\slfrac{g'}g(S_k)\to+\infty$ yields $\slfrac{g'}g(\zeta_k)\sim\slfrac{g'}g(S_k)$ u.c.s. and in turn \eqref{skzk}.
From this, \eqref{ukti}, \eqref{skzk}, \eqref{zeta_k}-\eqref{B_kA_k}, and \eqref{Lambda}, after using Fatou's Lemma, one obtains
\begin{equation}\label{utildacon}
	\begin{aligned}
		0<\int_{\R^2}\e^{b\tu}\dd x&\leq\liminf_{k\to+\infty}\int_{\Omega_k}\e^{\tslfrac{f'}f(S_k)\tslfrac{g'}g(\zeta_k(x))\tu_k(x)}\dd x\\
		&=\liminf_{k\to+\infty}\int_{\Omega_k}\frac{g(u_k(x_k+\lambda_k x))}{g(S_k)}\dd x\\
		&=\liminf_{k\to+\infty}B_k\int_{\Omega_k}g(u_k(x_k+\lambda_k x))\lambda_k^2\dd x\\
		&=\liminf_{k\to+\infty}\frac g{g'}(T_k)\intOmega g(u_k)\dd x\\
		&\leq\Lambda\lim_{k\to+\infty}\frac g{g'}(T_k)\,.
	\end{aligned}
\end{equation}
The contradiction is immediately reached in the \textit{non-Liouville} case \eqref{nonBM} since $\slfrac{g'}g(T_k)\to+\infty$. On the other hand, if \eqref{BM} is assumed, \eqref{utildacon} yields
\begin{equation}\label{utildacon_BM}
	0<\int_{\R^2}\e^{b\tu}\dd x\leq\frac\Lambda q\,.
\end{equation}
In this case, reasoning as in \eqref{5.1bis}, for a fixed $R>0$ we obtain that $\|\tu_k\|_{H^2(B_R(0))}\leq C_R$, hence $\tu_k\to\tu$ in $H^{1}_{loc}(\R^2)\cap C^{0,\alpha}_{loc}(\R^2)$ for all $\alpha\in(0,1)$. Moreover, by \eqref{scaling_uk_vk} with our choices $x_k=0$ and \eqref{B_kA_k}, we see that
\begin{equation*}
	\tv_k(x)-\tv_k(0)=\frac g{g'}(T_k)\left(v_k(\lambda_kx)-v_k(0)\right),
\end{equation*}
which, together with \eqref{bvk_to_bv} and \eqref{BM}, implies
\begin{equation*}
	v_k(\lambda_kx)-v_k(0)\to q\,\bv(x)\qquad\mbox{u.c.s..}
\end{equation*}
Since $v_k(0)=\|v_k\|_\infty\to+\infty$, we infer that $v_k(\lambda_kx)\to+\infty$ u.c.s., and in turn that $$\eta_k(x)=\theta_k(x)v_k(\lambda_kx)+(1-\theta_k(x))T_k\to+\infty\quad\ \mbox{u.c.s.}\,.$$
Hence $\slfrac{f'}f(\eta_k)\to p$ u.c.s. by \eqref{BM} and so
\begin{equation*}
	-\Delta\tu_k=\e^{\tslfrac{f'}f(\eta_k(x))\tslfrac{g'}g(T_k)\,\tv_k}\to0\quad\mbox{u.c.s.}\,.
\end{equation*}
Since $\nabla\tu_k\to\nabla\tu$ in $L^2_{loc}(\R^2)$, this implies that $\tu$ is weakly harmonic in $\R^2$. Moreover, $\tu(0)=\lim_{k\to+\infty}\tu_k(0)=0$ and $\tu\leq0$, therefore Liouville's theorem \cite[Theorem II]{SZ} implies $\tu\equiv0$ in $\R^2$, which contradicts \eqref{utildacon_BM}. We can therefore exclude that $\tv_k(0)\to-\infty$ u.c.s. and proceed with the other cases. 

\subsection{Case \texorpdfstring{$\boldsymbol{\tv_k(0)\to c\in\R}$}{2}}\label{Sec_4.2}
By \eqref{bvk_to_bv} $\tv_k$ is bounded u.c.s. and \eqref{tvk_ucs} holds. From Lemma \ref{f3_conseq}(\textit{iii}), we then have
\begin{equation*}
	\eta_k(x)=T_k\Big(1+\underbrace{\theta_k(x)}_{\in[0,1]}\underbrace{\frac{g'(T_k)}{g(T_k)T_k}}_{\to0}\tv_k\Big)\sim T_k \quad\mbox{u.c.s.}
\end{equation*}
as $k\to+\infty$, and moreover,
\begin{equation}\label{etak-Tk}
	\eta_k(x)-T_k=\theta_k(x)\frac{g'}g(T_k)\tv_k(x)\,.
\end{equation}
In order to show that $\slfrac{f'}f(\eta_k)\sim\slfrac{f'}f(T_k)$ u.c.s. as $k\to+\infty$, it is equivalent to prove that $\slfrac f{f'}(\eta_k)\sim\slfrac f{f'}(T_k)$ u.c.s.. By $(H_3)$, \eqref{etak-Tk}, $(H_b)$, and \eqref{tvk_ucs}, one has
\begin{equation}\label{proof_asympt_f}
	\begin{aligned}
		\left|\frac{\slfrac f{f'}(\eta_k(x))-\slfrac f{f'}(T_k)}{\slfrac f{f'}(T_k)}\right|&=\frac {f'}f(T_k)\left|\int_{T_k}^{\eta_k(x)}\left(\slfrac f{f'}\right)'\!(\xi)\dd\xi\right|\\
		&\le\left(\sup_{[T_k,\eta_k(x)]}\left(\frac f{f'}\right)'\right)
		\theta_k(x)\frac{f'}f(T_k)\frac{g'}g(T_k)|\tv_k(x)|\\
		&=o_k(1)(b+o_k(1))(\tv_k(0) +\bar v(x)+o_k(1))=o_k(1)\,,
	\end{aligned}
\end{equation}
u.c.s.. Therefore
\begin{equation*}
	-\Delta\tu_k(x)=\e^{\tslfrac{f'}{f}(T_k)\left(1+o_k(1)\right)\tslfrac{g'}{g}(T_k)\,\tv_k(x)}\leq C
\end{equation*}
which implies \eqref{ukti} by Lemmas \ref{LocalEstimates_Serrin} and \ref{Harnack_Serrin} and, in turn, \eqref{utildacon} follows as in the previous section. Therefore, in the \textit{non-Liouville} case \eqref{nonBM} we get a contradiction. On the other hand, in the \textit{Liouville case} \eqref{BM}, we only get the finite-mass condition \eqref{utildacon_BM}. As in the previous section, $\tu_k\to\tu$ in $C^{0,\alpha}_{loc}(\R^2)\cap H^1_{loc}(\R^2)$ for all $\alpha\in(0,1)$ where $\tu$ satisfies $-\Delta\tu=\e^{b\tv}$ in $\R^2$. The same analysis can be carried out for the other equation in \eqref{sys_k}, so the limit $(\tu,\tv)$ satisfies
\begin{equation*}
	\begin{cases}
		-\Delta\tu=\e^{b\tv}\quad&\mbox{in }\R^2\,,\\
		-\Delta\tv=\e^{b\tu}\quad&\mbox{in }\R^2\,,
	\end{cases}\qquad\mbox{with}\qquad\begin{cases}
		\int_{\R^2}\e^{b\tu}\dd x\leq\frac\Lambda q\,,\\
		\int_{\R^2}\e^{b\tv}\dd x\leq\frac\Lambda p\,.\\
	\end{cases}
\end{equation*}
Using the change of variables $U(x):=b\tu(\frac x{\sqrt b})$ and $V(x):=b\tv(\frac x{\sqrt b})$ it is easy to realise that $(U,V)$ satisfies 
\begin{equation}\label{sys_limit}
	\begin{cases}
		-\Delta U=\e^V\quad&\mbox{in }\R^2\,,\\
		-\Delta V=\e^U\quad&\mbox{in }\R^2\,,
	\end{cases}\qquad\mbox{with}\qquad\begin{cases}
		M_1:={\int_{\R^2}\e^U\dd x=b\int_{\R^2}\e^{b\tu}\dd x<+\infty}\,,\\
		M_2:={\int_{\R^2}\e^V\dd x=b\int_{\R^2}\e^{b\tv}\dd x<+\infty}\,.\\
	\end{cases}
\end{equation}

We can use the characterization of solutions for \eqref{sys_limit} in \cite[Theorem 1.4]{CSW} (see also \cite{CK}), according to which the mass identity
\begin{equation*}
    4\pi(M_1+M_2)=M_1M_2
\end{equation*}
holds. This identity may be rewritten as $M_2(M_1-4\pi)=4\pi M_1$, from which, recalling $M_1,M_2>0$, it is clear that
\begin{equation}\label{masses}
    M_1>4\pi\,.
\end{equation}
By symmetry, also $M_2>4\pi$ holds. As in \eqref{utildacon}, one gets
\begin{equation*}
	\frac{M_1}b=\int_{\R^2}\e^{b\tu}\dd x=\lim_{R\to+\infty}\int_{B_R(0)}\e^{b\tu}\dd x\leq\lim_{R\to+\infty}\liminf_{k\to+\infty}\frac g{g'}(T_k)\!\int_{B_{R\lambda_k}(0)}g(u_k(x))\dd x\,,
\end{equation*}
hence, by \eqref{BM},
\begin{equation}\label{conc8pi}
	\lim_{R\to+\infty}\liminf_{k\to+\infty}\int_{B_{R\lambda_k}(0)}g(u_k(x))\dd x\geq\frac{M_1}p\,.
\end{equation}
Since we have already observed that the only blow-up point for $(u_k)_k$ and $(v_k)_k$ is the origin, then $(u_k)_k, (v_k)_k$ are locally uniformly bounded in $\overline{B_1(0)} \setminus \{0\}$. Hence there exist $u,v$ s.t. $u_k\to u$ and $v_k\to v$ in $L^\tau_{loc}(\overline{B_1(0)} \setminus \{0\})$, $\tau>1$, and moreover, $(g_k)_k:=(g(u_k))_k$ and $(f_k)_k:=(f(v_k))_k$ are uniformly bounded in $L^1(B_1(0))$ by \eqref{Lambda}. Hence there exist $\mu,\nu\in\mathcal{M}(B_1(0))\cap L^\infty_{loc}(\overline{B_1(0)}\setminus\{0\})$ such that $g_k\rightharpoonup\mu$ and $f_k\rightharpoonup\nu$ in the sense of measures in $B_1(0)$, where $\mu, \nu \ge 0$ are singular at the origin, so we may write $\mu=A(x)\dd x+a_0\delta_0$, $\nu=B(x)\dd x+b_0\delta_0$ for some $a_0,b_0\in\R^+$ and $A,B\geq0$ in $L^1(B_1(0))$. Here $\delta_0$ denotes the Dirac measure.

We first show that
\begin{equation}\label{mu8pi}
	\mu\geq\frac{M_1}p\delta_0\qquad\mbox{and}\qquad\nu\geq\frac{M_2}q\delta_0\,.
\end{equation}
Since $g_k\rightharpoonup\mu$ in measure in $B_1(0)$, for all $\varphi\in C^\infty_0(B_1(0))$ we have 

\begin{equation}\label{phigk}
	\int_{B_1(0)}\varphi\dd\mu=\lim_{k\to+\infty}\int_{B_1(0)}\varphi g_k(x)\dd x=\lim_{k\to+\infty}\int_{B_1(0)}\varphi g(u_k(x))\dd x\,.
\end{equation}
Let now $t\in(0,1)$ and choose $\varepsilon>0$ so that $t(1+\varepsilon)<1$. Let $\varphi\in C^\infty_0(B_1(0))$ so that $\varphi\equiv1$ on $B_t(0)$, $\supp\,\varphi\subset B_{t(1+\varepsilon)}(0)$ and with values in $[0,1]$.
Then
\begin{equation}\label{estimateA}
	\begin{aligned}
		\int_{B_1(0)}\varphi\dd\mu&=\int_{B_{t(1+\varepsilon)}(0)}\varphi\dd\mu = \int_{B_{t(1+\varepsilon)}(0)\setminus B_t(0)}\varphi A(x)\dd x+\int_{B_t(0)}\dd\mu\,\\
		&\leq\int_{B_{t(1+\varepsilon)}(0)\setminus B_t(0)} A(x)\dd x+\int_{B_t(0)}\dd\mu\,.
	\end{aligned}
\end{equation}
By \eqref{conc8pi} for a fixed $\delta>0$, there exists $R_0>0$ such that for all $R\geq R_0$,
$$\frac{M_1}p-\delta\leq\liminf_{k\to+\infty}\int_{B_{R\lambda_k}(0)}g(u_k)\dd x\,.$$
Fixing $R=R_0$, and since $\lambda_k\to 0$, for $k$ large enough $R_0\lambda_k\leq t$, so $B_{R_0\lambda_k}(0)\subset B_t(0)$.	Hence by \eqref{phigk}
\begin{equation*}
	\begin{aligned}
		\int_{B_{t(1+\varepsilon)}(0)}\varphi d\mu &=\lim_{k\to+\infty} \int_{B_{t(1+\varepsilon)}(0)}\varphi g(u_k)\dd x \geq\liminf_{k\to+\infty} \int_{B_t(0)} \varphi g(u_k)\dd x\\
		&=\liminf_{k\to+\infty}\int_{B_t(0)}g(u_k)\dd x\geq\liminf_{k\to+\infty} \int_{B_{R_0\lambda_k}(0)}g(u_k(x))\dd x\geq\frac{M_1}p-\delta\,.
	\end{aligned}
\end{equation*}
By \eqref{estimateA} we therefore obtain
$$\int_{B_{t(1+\varepsilon)}(0)\setminus B_t(0)}A(x)\dd x+\int_{B_t(0)}\dd\mu\geq\frac{M_1}p-\delta\,.$$
Letting first $\delta\to0$ and then $\varepsilon\to0$, this yields
$$\int_{B_t(0)}\dd\mu\geq\frac{M_1}p\,,$$
which in turn, by the arbitrariness of $t\in(0,1)$, implies $\mu\geq\frac{M_1}p\delta_0$. We can argue similarly with $\nu$ and conclude that \eqref{mu8pi} holds. Let now $\Phi,\Psi$ be the distributional solutions of
\begin{equation*}
	\begin{cases}
		-\Delta\Phi=\mu&\mbox{in }B_1(0)\,,\\
		\Phi=0\quad&\mbox{on }\dB_1(0)\,,
	\end{cases}\quad\ \mbox{and}\quad\ \begin{cases}	
		-\Delta\Psi=\nu&\mbox{in }B_1(0)\,,\\
		\Psi=0\quad&\mbox{on }\dB_1(0)\,.
	\end{cases}
\end{equation*}
Since 
\begin{equation*}
	\begin{cases}
		-\Delta u_k=f(v_k)=f_k\rightharpoonup\nu&\mbox{in }B_1(0)\,,\\
		u_k=0\quad&\mbox{on }\dB_1(0)\,,
	\end{cases}\quad\ \mbox{and}\quad\ \begin{cases}	
		-\Delta v_k=g(u_k)=g_k\rightharpoonup\mu&\mbox{in }B_1(0)\,,\\
		v_k=0\quad&\mbox{on }\dB_1(0)\,,
	\end{cases}
\end{equation*}
then $u_k\to\Psi$ and $v_k\to\Phi$ in $W^{1,\sigma}(B_1(0))$ for $\sigma\in[1,2)$ (see \cite[Proposition 5.1]{Ponce}) and, by \eqref{mu8pi} and the maximum principle,
\begin{equation}\label{PhiPsi}
	\Phi\geq\frac{M_1}pG_B(\cdot,0)=\frac{M_1}{2\pi p} \log\frac1{|\cdot|}\qquad\mbox{and}\qquad\Psi\geq\frac{M_2}{2\pi q}\log\frac1{|\cdot|}\,,
\end{equation}
where $G_B$ stands for the Green's function of $-\Delta$ with Dirichlet boundary conditions in the ball $B_1(0)$. By \cite[Lemma 1.1]{MR} and \eqref{BM}, we know that for all $\varepsilon>0$, there exist constants $C_\varepsilon, D_\varepsilon$ such that
\begin{equation}\label{MRLem1.1}
	D_\varepsilon\e^{p(1-\varepsilon)t}-C_\varepsilon\leq f(t) \leq D_\varepsilon\e^{p(1+\varepsilon)t}+C_\varepsilon\,.
\end{equation}
On the one hand, by \eqref{PhiPsi} we get
\begin{equation}\label{one_hand}
	\int_{B_1(0)}\e^{p(1-\varepsilon)\Phi}\dd x\geq \int_{B_1(0)}\e^{\frac{M_1}{2\pi}(1-\varepsilon)\log\frac1{|x|}}\dd x=2\pi\int_0^1\frac{\dd\rho}{\rho^{\frac{M_1}{2\pi}(1-\varepsilon)-1}}=+\infty\,,
\end{equation}
provided $\varepsilon > 0$ is small enough by \eqref{masses}. On the other hand, by Fatou's lemma and \eqref{MRLem1.1}, one has
\begin{equation*}
	\begin{aligned}
		\int_{B_1(0)}\e^{p(1-\varepsilon)\Phi}\dd x&= \int_{B_1(0)}\lim_{k\to+\infty}\e^{p(1-\varepsilon)v_k}\dd x\leq \liminf_{k\to+\infty}\int_{B_1(0)}\e^{p(1-\varepsilon)v_k}\dd x\\
		&\leq\frac1{D_\varepsilon}\liminf_{k\to+\infty}\int_{B_1(0)}f(v_k)\dd x +\frac{C_\varepsilon}{D_\varepsilon}|B_1(0)|\leq\frac{\Lambda + |B_1(0)|C_\varepsilon}{D_\varepsilon}<+\infty\,,
	\end{aligned}
\end{equation*}
which clearly contradicts \eqref{one_hand}, and eventually concludes the argument.

\subsection{Case \texorpdfstring{$\boldsymbol{\tv_k(0)\to+\infty}$}{3}}\label{+infty}
This time the easiest setting is the Liouville. Indeed, supposing \eqref{BM}, then
\begin{equation*}
	\eta_k(x)=\underbrace{T_k}_{\to+\infty}+\underbrace{\theta_k(x)}_{\in[0,1]}\underbrace{\frac{g'}{g}(T_k)}_{\to q}\underbrace{\tv_k(x)}_{\to+\infty}\to+\infty\quad\ \mbox{u.c.s.}\,,
\end{equation*}
so $\slfrac{f'}f(\eta_k(x))\to p$ u.c.s.. This, together with $\slfrac{g'}g(T_k)\to q$, similarly to \eqref{utildacon} implies
\begin{equation*}
	\frac\Lambda p\leftarrow\Lambda\frac f{f'}(S_k)\geq\int_{B_R(0)}\e^{\tslfrac{f'}f(\eta_k(x))\tslfrac{g'}g(T_k)\,\tv_k(x)}\dd x\to+\infty\,,
\end{equation*}
for any $R>0$ chosen arbitrarily, a contradiction. Suppose then from now on that \eqref{nonBM} holds. Of course, also in this case we can rely on the inequality
\begin{equation}\label{ineq_key}
	\int_{B_R(0)}\e^{\tslfrac{f'}f(\eta_k(x))\tslfrac{g'}g(T_k)\,\tv_k(x)}\dd x\leq\Lambda\frac f{f'}(S_k)\,.
\end{equation}
Here, our analysis branches in two subcases. For the remaining part of the section, estimates on $B_R(0)$ are understood for fixed $R>0$ and sufficiently large $k$. Constants indexed by $R$ are independent of $k$ and only depend on the choice of $R$.

To conclude our argument in this case we need the additional assumptions $(H_4)$, $(H_5)$, and one between $(H_6)$ and $(H_6')$. Note that, although $(H_4)$ has been used also in \eqref{proof_asympt_g} and \eqref{proof_asympt_f}, it can be easily bypassed there.

\subsubsection{Subcase 1\texorpdfstring{: $\boldsymbol{0<\displaystyle\frac{\tv_k(0)}{T_kB_k}\leq C}$ for some $\boldsymbol{C>0}$}{}}

In this case clearly $\eta_k(x)\geq T_k$ holds u.c.s. by \eqref{eta_k}. Aiming at finding a contradiction by means of \eqref{ineq_key}, we need to bound from below the argument of the exponential.
Using $(H_5)$, \eqref{bvk}, \eqref{bvk_to_bv}, the definition of $\eta_k$, and since $\tv_k(x)\leq\tv_k(0)$ because it is decreasing along the radius, we estimate
\begin{equation}\label{pk_qk}
	\begin{aligned}
		\frac{f'}f(\eta_k(x))\frac{g'}g(T_k)\,\tv_k(x)&\geq\frac{f'}f(T_k)\frac{g'}g(T_k)\frac{T_k}{\eta_k(x)}\tv_k(x)\\
		&\geq\frac{f'}f(T_k)\frac{g'}g(T_k)\frac{\tv_k(x)}{1+\frac{\tv_k(x)}{T_kB_k}}\geq\frac{f'}f(T_k)\frac{g'}g(T_k)\frac{\tv_k(x)}{1+\frac{\tv_k(0)}{T_kB_k}}\\
		&=q_k\tv_k(x)=q_k\tv_k(0)+q_k\bv_k(x)\,,
	\end{aligned}
\end{equation}
where $q_k$ is defined as
\begin{equation}\label{qk}
	q_k:=\frac{f'}f(T_k)\frac{g'}g(T_k)\frac1{1+\frac{\tv_k(0)}{T_kB_k}}\,.
\end{equation}
Hence, from \eqref{ineq_key}, since $\bv_k\leq0$ and $0\leq q_k\leq b(1+o_k(1))$,
\begin{equation}\label{4.22bis}
	\Lambda\frac f{f'}(S_k)\geq\e^{q_k\tv_k(0)}\int_{B_R(0)}\e^{q_k\bv_k(x)}\dd x\geq\CR\,\e^{q_k\tv_k(0)},
\end{equation}
where
\begin{equation}\label{C_R}
\CR:=|B_R(0)|\,\e^{b\left(\min_{B_R(0)}\bv-1\right)}.
\end{equation}
This yields an estimate from above for the growth of $\tv_k(0)$:
\begin{equation}\label{estimate_1b}
	q_k\tv_k(0)\leq\log\left(\frac\Lambda{\CR}\right)+\log\left(\frac f{f'}(S_k)\right).
\end{equation}
In order to reach a contradiction, we need to obtain an estimate from below for $\tv_k(0)$, and this is accomplished by means of the other energy estimate \eqref{utildacon}, that is
\begin{equation}\label{utildacon_conseq}
	\int_{B_R(0)}\e^{\tslfrac{f'}f(S_k)\tslfrac{g'}g(\zeta_k(x))\tu_k(x)}\dd x\leq\Lambda\frac g{g'}(T_k)\,,
\end{equation} 
which is in terms of $\tu_k$, so we need to bound it from below. Exploiting the radial symmetry, and denoting by $r$ the radial coordinate, since $\tu_k(0)=0$ we have
\begin{equation}\label{stima_TdD}
	\begin{aligned}
		2\pi|\tu_k(r)|&=2\pi\left(\tu_k(0)-\tu_k(r)\right)=-2\pi\int_0^r\tu_k'(s)\dd s=-\int_0^r\frac{|\partial B_s(0)|}s\,\partial_\nu{\tu_k}{|_{\partial B_s(0)}}\dd s\\
		&=-\int_0^r\frac1s\Bigg(\int_{\partial B_s(0)}\partial_\nu\tu_k\Bigg)\dd s=\int_0^r\frac1s\Bigg(\int_{B_s(0)}-\Delta\tu_k(x)\dd x\Bigg)\dd s\\
		&=\int_0^r\frac1s\Bigg(\int_{B_s(0)}\e^{\tslfrac{f'}f(\eta_k(x))\tslfrac{g'}g(T_k)\,\tv_k(x)}\dd x\Bigg)\dd s\,.
	\end{aligned}
\end{equation}
Since $\eta_k(x)\geq T_k$ u.c.s., it is easy to see by $(H_4)$ that
\begin{equation*}
	\frac{f'}f(\eta_k(x))\frac{g'}g(T_k)\tv_k(x)\leq\frac{f'}f(T_k)\frac{g'}g(T_k)\tv_k(x)\,.
\end{equation*}
Defining
\begin{equation}\label{gamma_k}
	\gamma_k:=\frac{f'}f(T_k)\frac{g'}g(T_k)\geq0\,,
\end{equation}
since $\bv_k(x)\leq0$, from \eqref{stima_TdD} and \eqref{tvk_ucs}, we infer
\begin{equation*}
	\begin{aligned}
		2\pi|\tu_k(r)|&\leq\int_0^r\frac1s\Bigg(\int_{B_s(0)}\e^{\gamma_k\tv_k(x)}\dd x\Bigg)\dd s=\e^{\gamma_k\tv_k(0)}\int_0^r\frac1s\Bigg(\int_{B_s(0)}\e^{\gamma_k\bv_k(x)}\dd x\Bigg)\dd s\\
		&\leq\e^{\gamma_k\tv_k(0)}\int_0^r\frac1s|B_s(0)|\dd s=\frac{\pi r^2}2\,\e^{\gamma_k\tv_k(0)}\,.
	\end{aligned}
\end{equation*}
By \eqref{4.22bis} this yields
\begin{equation}\label{estimate_2a}
	-\tu_k(r)=|\tu_k(r)|\leq\e^{\gamma_k\tv_k(0)}\frac{r^2}4\leq\left(\frac\Lambda{\CR}\frac f{f'}(S_k)\right)^{\frac{\gamma_k}{q_k}}\frac{r^2}4\,.
\end{equation}

Note that, by \eqref{zeta_k} one has $\zeta_k(x)\leq S_k$ because $\tu_k\leq0$. Hence, it would be nice to apply $(H_4)$ and get $\frac{g'}g(\zeta_k(x))\leq\frac{g'}g(S_k)$, so that, combining this with the estimate from below \eqref{estimate_2a} (remember that $\tu_k\leq0$), we would manage to bound from below the left-hand side of \eqref{utildacon_conseq}. However, $(H_4)$ holds only for large $s$, and so one is first compelled to show that
\begin{equation}\label{zk_infty}
	\zeta_k(x)\to+\infty\quad\mbox{u.c.s.}
\end{equation}
as $k\to+\infty$. We claim in fact that for any compact set $K\subset\R^2$ there exists $C_1=C_1(K)>0$ such that 
\begin{equation*}
	\frac{\zeta_k(x)}{S_k}\geq C_1\quad\mbox{for all}\ \,x\in K\ \mbox{and }k\mbox{ large}\,.
\end{equation*}
Let us proceed by contradiction, and assume that, fixing an arbitrary $R>0$, one finds $(\xi_k)_k\subset B_R(0)$ such that $\frac{\zeta_k(\xi_k)}{S_k}\to0$ as $k\to+\infty$. By \eqref{zeta_k}, this means that
\begin{equation*}
	(1-\thetabar_k(\xi_k))+\thetabar_k(\xi_k)\frac{u_k(\lambda_k\xi_k)}{S_k}\to0\,.
\end{equation*}
Since both quantities are positive, necessarily $\thetabar_k(\xi_k)\to1$. Then, again by \eqref{zeta_k}, one infers
$$|\tu_k(\xi_k)|=-\tu_k(\xi_k)=(1+o_k(1))S_k\frac f{f'}(S_k)\,.$$
This, together with \eqref{estimate_2a} for $r=|\xi_k|\leq R$, and \eqref{qk} and \eqref{gamma_k}, for $k$ large yields
\begin{equation*}
	S_k\frac f{f'}(S_k)\leq\frac{R^2}{2}\left(\frac\Lambda{\CR}\frac f{f'}(S_k)\right)^{1+\frac{\tv_k(0)}{T_kB_k}}\leq\overline C_R\left(\frac f{f'}(S_k)\right)^{1+\frac{\tv_k(0)}{T_kB_k}}
\end{equation*}
for some $\overline C_R>0$, since we are in Subcase 1. Hence, passing to the logarithms, and noticing that $B_k=\slfrac g{g'}(T_k)=b^{-1}(1+o_k(1))\slfrac{f'}f(T_k)$ by $(H_b)$ for large $k$, one infers
\begin{equation*}
	\tv_k(0)\geq\frac{1+o_k(1)}b\,T_k\frac{f'}f(T_k)\left(\log\left(\frac f{f'}(S_k)\right)\right)^{-1}\!\!\left(\log S_k-\log\overline C_R\right).
\end{equation*}
Combining this with \eqref{estimate_1b} and $(H_b)$, yields
\begin{equation*}
	\begin{aligned}
		2\log\left(\frac f{f'}(S_k)\right)&\geq\log\left(\frac\Lambda{\CR}\right)+\log\left(\frac f{f'}(S_k)\right)\geq q_k\tv_k(0)\\
		&\geq\frac{b(1+o_k(1))}{1+\frac{\tv_k(0)}{T_kB_k}}\,\tv_k(0)\geq\frac b{2(1+C)}\,\tv_k(0)\\
		&\geq\frac1{4(1+C)}\,T_k\frac{f'}f(T_k)\left(\log\left(\frac f{f'}(S_k)\right)\right)^{-1}\log S_k\,,
	\end{aligned}
\end{equation*}
that is,
\begin{equation*}
	8(1+C)\geq\frac{T_k\slfrac{f'}f(T_k)}{\left(\log\slfrac f{f'}(S_k)\right)^2}\,\log S_k\,.
\end{equation*}
This is however in contradiction with Lemma \ref{log_contradiction}, from which the ratio on the right-hand side is bounded away from zero. Thus, \eqref{zk_infty} is proved, and we can proceed with the argument which exploits \eqref{utildacon_conseq}. Indeed, by $(H_4)$
now we obtain $\frac{g'}g(\zeta_k(x))\leq\frac{g'}g(S_k)$, and in turn $\frac{g'}g(\zeta_k(x))\tu_k(x)\geq\frac{g'}g(S_k)\tu_k(x)$. Hence, from \eqref{utildacon_conseq} using \eqref{estimate_2a} one gets
\begin{equation*}
	\begin{aligned}
		\Lambda\frac g{g'}(T_k)&\geq\int_{B_R(0)}\e^{\frac{f'}f(S_k)\frac{g'}g(S_k)\tu_k(x)}\dd x\geq\int_{B_R(0)}\e^{-\frac{f'}f(S_k)\frac{g'}g(S_k)\,\e^{\gamma_k\tv_k(0)}\frac{|x|^2}4}\dd x\\
		&=:\int_{B_R(0)}\e^{-E_k|x|^2}\dd x=\frac\pi{E_k}\left(1-\e^{-R^2E_k}\right)=\frac\pi{E_k}\left(1+o_k(1)\right)\geq\frac\pi{2E_k}\,,
	\end{aligned}
\end{equation*}
because
$$E_k:=\frac14\frac{f'}f(S_k)\frac{g'}g(S_k)\,\e^{\tslfrac{f'}f(T_k)\tslfrac{g'}g(T_k)\tv_k(0)}\to+\infty$$
by $(H_b)$ and $\tv_k(0)\to+\infty$. Therefore, by \eqref{4.22bis}
\begin{equation*}
	\begin{aligned}
		\frac{2\Lambda}\pi&\geq\frac{g'}g(T_k)E_k^{-1}=\frac{g'}g(T_k)\,\frac4{\slfrac{f'}f(S_k)\slfrac{g'}g(S_k)}\,\e^{-\frac{\gamma_k}{q_k}\cdot q_k\tv_k(0)}\\
		&\geq\frac2b\,\frac{g'}g(T_k)\left(\frac\Lambda{\CR}\frac f{f'}(S_k)\right)^{-\frac{\gamma_k}{q_k}}=\frac2b\left(\frac\Lambda{\CR}\right)^{-\frac{\gamma_k}{q_k}}\left(\frac{f'}f(S_k)\right)^{\frac{\gamma_k}{q_k}}\frac{g'}g(T_k)\,,
	\end{aligned}
\end{equation*}
which can be rewritten as
\begin{equation}\label{*}
	\frac{b\Lambda}\pi\left(\frac\Lambda{\CR}\right)^{\frac{\gamma_k}{q_k}}\geq\left(\frac{f'}f(S_k)\right)^{\frac{\gamma_k}{q_k}-1}\frac{f'}f(S_k)\frac{g'}g(T_k)\geq b(1+o_k(1))\left(\frac{f'}f(S_k)\right)^{\frac{\gamma_k}{q_k}-1}
\end{equation}
by $(H_4)$ and Lemma \ref{t>s_Lemma}. Notice that the exponent
\begin{equation*}
	\frac{\gamma_k}{q_k}-1=\frac{\tv_k(0)}{T_kB_k}>0
\end{equation*}
by \eqref{qk} and \eqref{gamma_k}, and that $\slfrac{f'}f(S_k)\to0$, therefore the behaviour of the right-hand side of \eqref{*} is different according to the following two alternatives:
\begin{equation}\label{I_II}
	\frac{\tv_k(0)}{T_kB_k}\log\left(\frac{f'}f(S_k)\right)\to \begin{cases}
		c\leq0&\quad\,\mbox{(I)}\,,\\
		-\infty&\quad\,\mbox{(II)}\,,
	\end{cases}
\end{equation}
possibly up to a subsequence.

If (\ref{I_II}-I) holds, then
\begin{equation}\label{consequence_I}
	\left(\frac{f'}f(S_k)\right)^{\frac{\gamma_k}{q_k}-1}=\e^{\left(\frac{\gamma_k}{q_k}-1\right)\log\left(\frac{f'}f(S_k)\right)}=\e^{\frac{\tv_k(0)}{T_kB_k}\log\left(\frac{f'}f(S_k)\right)}\to C>0
\end{equation}
as $k\to+\infty$. Hence, from \eqref{*}, $(H_b)$, $(H_4)$, and Lemma \ref{t>s_Lemma},
\begin{equation}\label{**}
	\frac{b\Lambda}\pi\left(\frac\Lambda{\CR}\right)^{\frac{\gamma_k}{q_k}}\geq\left(\frac{f'}f(S_k)\right)^{\frac{\gamma_k}{q_k}-1}\frac{f'}f(S_k)\frac{g'}g(S_k)\to Cb>0\,,
\end{equation}
so in order to reach a contradiction, we need that the constant $\CR$, defined in \eqref{C_R} and appearing on the left-hand side, diverges to $+\infty$ as $R\to+\infty$. Note, indeed, that so far the radius $R$ was arbitrary. Since $\CR$ depends on $\min_{B_R(0)}\bar v$, which is the limit of $\min_{B_R(0)} \bar v_k$, and $\bar v_k $ satisfies \eqref{eq_bvk}, we need to investigate first the global bounds of $\tu_k$. By \eqref{pk_qk} and \eqref{stima_TdD}, for $x\in B_R(0)$, $r=|x|$, and $k$ large,
\begin{equation}\label{better_from_below}
	\begin{aligned}
		-2\pi\tu_k(x)&=\int_0^r\frac1s\Bigg(\int_{B_s(0)}\e^{\tslfrac{f'}f(\eta_k(x))\tslfrac{g'}g(T_k)\tv_k(x)}\dd x\Bigg)\dd s\geq\e^{q_k\tv_k(0)}\int_0^r\frac1s\Bigg(\int_{B_s(0)}\e^{q_k\bv_k(x)}\dd x\Bigg)\dd s\,.
	\end{aligned}
\end{equation}
Let us fix $R_0>0$. If $s\geq R_0$, then
\begin{equation}\label{I_ineq}
	\int_{B_s(0)}\e^{q_k\bv_k(x)}\dd x\geq\int_{B_{R_0}(0)}\e^{q_k\bv_k(x)}\dd x\geq\int_{B_{R_0}(0)}\e^{(b+o_k(1))\bv_k(x)}\dd x\to\int_{B_{R_0}(0)}\e^{b\bv(x)}\dd x=:M_0>0
\end{equation}
by \eqref{bvk_to_bv}, by recalling that $0<q_k<b+o_k(1)$ and $\bv_k\leq0$. On the other hand, if $s<R_0$ then
\begin{equation}\label{II_ineq}
	\begin{aligned}
		\int_{B_s(0)}\e^{q_k\bv_k(x)}\dd x&\geq\int_{B_s(0)}\left(\e^{(b+o_k(1))\bv_k(x)}-\e^{b\bv(x)}\right)\dd x+\int_{B_s(0)}\e^{b\bv(x)}\dd x\\
		&\geq-\int_{B_s(0)}\max_{x\in B_{R_0}(0)}\left|\e^{(b+o_k(1))\bv_k(x)}-\e^{b\bv(x)}\right|\dd x+\int_{B_s(0)}\e^{b\bv(x)}\dd x\\
		&\geq\left(-M_{k,R_0}+\e^{b\min_{B_{R_0}(0)}\bv}\right)\pi s^2\geq\varepsilon_0\pi s^2
	\end{aligned}
\end{equation}
for some $\varepsilon_0>0$, since $M_{k,R_0}:=\max_{x\in B_{R_0}(0)}\left|\e^{(b+o_k(1))\bv_k(x)}-\e^{b\bv(x)}\right|\to0$ as $k\to+\infty$ since $R_0$ is fixed. Therefore we may combine \eqref{I_ineq} and \eqref{II_ineq} by writing
$$\int_{B_s(0)}\e^{q_k\overline v_k(x)}\dd x\geq\varepsilon_1\,\frac{\pi s^2}{1+s^2}\,,$$
for $0<\varepsilon_1<\min\left\{\varepsilon_0,\,\frac{M_0}\pi\right\}$ which holds for all $s>0$. Thus, \eqref{better_from_below} becomes
\begin{equation*}
	-2\pi\tu_k(x)\geq\varepsilon_1\e^{q_k\tv_k(0)}\int_0^r\frac{\pi s}{1+s^2}\dd s=\e^{q_k\tv_k(0)}\frac{\varepsilon_1\pi}2\log(1+r^2)\,,
\end{equation*}
which yields
\begin{equation}\label{better_from_below_2}
	\tu_k(x)\leq-\frac{\varepsilon_1}4\log(1+|x|^2)\,\e^{q_k\tv_k(0)}
\end{equation}
for all $x\in\Omega_k$. Since $(H_b)$ and the boundedness of $\frac{\tv_k(0)}{T_kB_k}$ imply that $q_k\not\to0$, then $q_k\tv_k(0)\to+\infty$, and in turn
\begin{equation*}
	\tu_k(x)\to-\infty\quad\mbox{as}\ \,k\to+\infty
\end{equation*}
for all $x\neq0$. We transfer this information to obtain properties of $\bv$ by \eqref{eq_bvk}. Arguing as in \eqref{stima_TdD}, we have
\begin{equation}\label{stima_vkt}
	\begin{aligned}
		2\pi|\bv_k(r)|&=2\pi\left(\bv_k(0)-\bv_k(r)\right)=\int_0^r\frac1s\Bigg(\int_{B_s(0)}-\Delta\bv_k(x)\dd x\Bigg)\dd s\\
		&=\int_0^r\frac1s\Bigg(\int_{B_s(0)}\e^{\tslfrac{f'}f(S_k)\tslfrac{g'}g(\zeta_k(x))\,\tu_k(x)}\dd x\Bigg)\dd s\,,
	\end{aligned}
\end{equation}
so we need to analyse the exponent on the right. Recalling $S_k\geq\zeta_k(x)$ and $\tu_k\leq0$, by $(H_5)$ and \eqref{zk_infty} one infers
\begin{equation}\label{ineqq}
	\begin{aligned}
		\frac{f'}f(S_k)\frac{g'}g(\zeta_k(x))\tu_k(x)&\leq\frac{f'}f(S_k)\frac{g'}g(S_k)\frac{\zeta_k(x)}{S_k}\tu_k(x)
	\end{aligned}
\end{equation}
and by \eqref{zeta_k}, \eqref{B_kA_k}, \eqref{estimate_2a}, and \eqref{4.22bis}, one estimates
\begin{equation}\label{ineqq_2}
	\frac{\zeta_k(x)}{S_k}=1+\overline\theta_k(x)\frac{\tu_k(x)}{A_kS_k}\geq1-\overline\theta_k(x)\frac{\e^{\gamma_k\tv_k(0)}}{A_kS_k}\,\frac{r^2}4\geq1-\left(\frac\Lambda{\CR}\right)^\frac{\gamma_k}{q_k}\frac{r^2}4\left(\frac f{f'}(S_k)\right)^{\frac{\gamma_k}{q_k}-1}\frac1{S_k}\to1
\end{equation}
as $k\to+\infty$ by \eqref{consequence_I}. Therefore, by \eqref{ineqq}, \eqref{ineqq_2}, and \eqref{better_from_below_2},
\begin{equation}\label{ineqq_3}
	\begin{aligned}
		\frac{f'}f(S_k)\frac{g'}g(\zeta_k(x))\tu_k(x)&\leq\frac{f'}f(S_k)\frac{g'}g(S_k)\frac{\zeta_k(x)}{S_k}\left(-\frac{\varepsilon_1}4\log(1+r^2)\e^{q_k\tv_k(0)}\right)\\
		&\leq-\left(b+o_k(1)\right)\frac{\varepsilon_1}4\log(1+|x|^2)\e^{q_k\tv_k(0)}\to-\infty
	\end{aligned}
\end{equation}
for all $x\neq0$. By \eqref{ineqq_3} and dominated convergence theorem ($\tu_k\le 0$, hence $\e^{\tslfrac{f'}f(S_k)\tslfrac{g'}g(\zeta_k(x))\,\tu_k(x)}\le 1$), we get that, for any $s\in[0,r]$, 
$$\lim_{k\to +\infty } \frac1s\int_{B_s(0)}\e^{\tslfrac{f'}f(S_k)\tslfrac{g'}g(\zeta_k(x))\,\tu_k(x)}\dd x =0\,.$$
Since
$$\frac1s\int_{B_s(0)}\e^{\tslfrac{f'}f(S_k)\tslfrac{g'}g(\zeta_k(x))\,\tu_k(x)}\dd x \le \pi s\,,$$
we may apply dominated convergence again and get directly from \eqref{stima_vkt} that $\bv_k(r)\to 0$ as $k\to +\infty$. In particular $\bv = 0$ in $\R^2$ by \eqref{bvk_to_bv}. As a consequence the constant $\CR$ defined in \eqref{C_R} behaves like $R^2$. Observe that under (\ref{I_II}-I) necessarily $\tfrac{\gamma_k}{q_k}\to1$. Hence, taking the limit as $k\to\infty$ in \eqref{**}, we get $\Lambda^2\ge C\pi\CR\,.$ Therefore, by taking $R\to+\infty$, we reach the desired contradiction. This concludes our analysis of Subcase 1 under the condition (\ref{I_II}-I).

\vskip0.2truecm
Let us now consider Subcase 1 under the complementary condition (\ref{I_II}-II), namely
\begin{equation*}
	\frac{\tv_k(0)}{T_kB_k}\log\left(\frac f{f'}(S_k)\right)\geq M
\end{equation*}
for any arbitrary $M>0$, for all $k$ large. Here the strategy resembles the one used to prove \eqref{zk_infty}. Indeed, starting again from \eqref{ineq_key}, by $(H_5)$, \eqref{eta_k}, and \eqref{bvk_to_bv}, one gets

\begin{equation*}
    \begin{aligned}
    	\Lambda\frac f{f'}(S_k)&\geq\int_{B_R(0)}\e^{\tslfrac{f'}f(T_k)\tslfrac{g'}g(T_k)\frac{T_k}{\eta_k(x)}\,\tv_k(x)}\dd x\\
        &\geq C_R\,\e^{\tslfrac{f'}f(T_k)\tslfrac{g'}g(T_k)\frac{\tv_k(0)}{1+C}}\geq C_R\,\e^{\tslfrac{f'}f(T_k)\frac{M}{1+C}\,T_k\left(\log\left(\tslfrac f{f'}(S_k)\right)\right)^{-1}},
    \end{aligned}
\end{equation*}
where $C_R$ is a constant depending on $R$. Hence, applying the logarithm on both sides, one obtains
\begin{equation*}
	2\log\left(\frac f{f'}(S_k)\right)\geq\log\left(\frac\Lambda{C_R}\right)+\log\left(\frac f{f'}(S_k)\right)\geq\frac{MT_k\slfrac{f'}f(T_k)}{(1+C)\log\slfrac f{f'}(S_k)}\,,
\end{equation*}
since $\slfrac{f'}f(S_k)\to0$, that is
\begin{equation}\label{log_contradiction_new}
	\frac{2(1+C)}M\geq\frac{T_k\slfrac{f'}f(T_k)}{\left(\log\slfrac f{f'}(S_k)\right)^2}\,.
\end{equation}
Since $M$ is arbitrary, this contradicts Lemma \ref{log_contradiction} by means of assumptions $(H_6)$ or $(H_6')$, and concludes our analysis of Subcase 1 also under assumption (\ref{I_II}-II).

\subsubsection{Subcase 2\texorpdfstring{: $\boldsymbol{\displaystyle\frac{\tv_k(0)}{T_kB_k}\to+\infty}$}{}}

As in Subcase 1, clearly $\eta_k(x)\geq T_k$ u.c.s., hence \eqref{pk_qk} holds, and so by \eqref{bvk}-\eqref{bvk_to_bv} we get
\begin{equation*}
	\begin{aligned}
		\frac{f'}f(\eta_k(x))\frac{g'}g(T_k)\tv_k(x)&\geq\frac{f'}f(T_k)\frac{g'}g(T_k)\frac{\tv_k(0)}{1+\frac{\tv_k(0)}{T_kB_k}}-C_R\,.
	\end{aligned}
\end{equation*}
By the elementary estimate
$$\frac1{1+y}=\frac1y\cdot\frac1{1+\frac1y}\geq\frac1y\left(1-\frac1y\right)$$
and \eqref{B_kA_k}, one ends up with
\begin{equation*}
	\begin{aligned}
		\frac{f'}f(\eta_k(x))\frac{g'}g(T_k)\tv_k(x)&\geq\frac{f'}f(T_k)\frac{g'}g(T_k)T_kB_k\left(1-\frac{T_kB_k}{\tv_k(0)}\right)-C_R\\
		&=\frac{f'}f(T_k)T_k\left(1-\frac{T_kB_k}{\tv_k(0)}\right)-C_R=\frac{f'}f(T_k)T_k(1+o_k(1))-C_R\,,
	\end{aligned}
\end{equation*}
recalling that we are supposing $\frac{\tv_k(0)}{T_kB_k}\to+\infty$. By \eqref{ineq_key} this yields
\begin{equation}\label{ineq_key_delta}
	\Lambda\geq\frac{f'}f(S_k)\int_{B_R(0)}\e^{\tslfrac{f'}f(\eta_k(x))\tslfrac{g'}g(T_k)\tv_k(x)}\dd x\geq|B_R(0)|\e^{-C_R}\,\e^{(1-\varepsilon)\tslfrac{f'}f(T_k)T_k}\frac{f'}f(S_k)\,.
\end{equation}
Now, if we suppose that $(H_6)$ holds, it is easy to find a contradiction. Indeed, choosing $\varepsilon=1-\frac1a\in(0,1)$, where $a>1$ is defined in $(H_6)$, by $(H_5)$ and Lemma \ref{t>s_Lemma} one gets
\begin{equation*}
	\begin{aligned}
		\frac{\Lambda\e^{C_R}}{|B_R(0)|}&\geq\e^{(1-\varepsilon)\tslfrac{f'}f(T_k)T_k}\frac{f'}f(S_k)\geq\e^{(1-\varepsilon)\tslfrac{f'}f(S_k)S_k}\frac{f'}f(S_k)\\
		&\geq\e^{(\log S_k)^2}\frac{f'}f(S_k)\geq S_k\frac{f'}f(S_k)\to+\infty
	\end{aligned}
\end{equation*}
by Lemma \ref{f3_conseq}(\textit{ii}). 

On the other hand, under $(H_6')$, again taking $\varepsilon=1-\frac1a$,
\begin{equation*}
	\frac{\Lambda\e^{C_R}}{|B_R(0)|}\geq\e^{(1-\varepsilon)\tslfrac{f'}f(T_k)T_k}\frac{f'}f(S_k)\geq\e^{(\log\log F(T_k))^2}\frac{f'}f(S_k)\geq\log F(T_k)\cdot\frac{f'}f(S_k)\to+\infty
\end{equation*}
by Lemma \ref{t>s_Lemma}. These contradictions conclude both Subcase 2 and the whole proof of Theorem \ref{Thm_Uniform_bound}.

\section{Existence: Proof of Theorem \ref{Thm_Existence}}\label{Sec_Existence}

We follow the approach of \cite{QS}, which is based on the fixed-point (or Leray-Schauder) index theory, and we refer to \cite{AdF} for its definition and basic properties. In what follows, in order to apply Theorem \ref{Thm_Uniform_bound}, we consider $\Omega=B_1(0)$. However, the fixed-point index argument in the proof of Theorem \ref{Thm_Existence} also applies to general smooth bounded planar domains, provided the corresponding a priori bounds hold uniformly for the family of auxiliary problems introduced below.
\vskip0.2truecm
Let $K$ be the cone of functions in $C(\Omegabar)\times C(\Omegabar)$ which are non-negative in $\overline{\Omega}$ and denote by $S$ the solution operator of the system
\begin{equation*}
	\begin{cases}
		-\Delta u=\phi\quad&\mbox{in }\Omega\,,\\
		-\Delta v=\psi\quad&\mbox{in }\Omega\,,\\
		u=v=0\quad&\mbox{on }\dOmega\,,
	\end{cases}
\end{equation*}
namely $S(\phi,\psi)=(u,v)$. Note that $S$ is a linear compact operator from $C(\Omegabar)\times C(\Omegabar)$ into itself. Hence solutions of \eqref{sys} correspond to fixed points of the map $T:(u,v)\mapsto T(u,v):=S(f(v),g(u))$, which is compact by $(H_1)$. Moreover, $T(K)\subseteq K$ by the maximum principle. For $W\subset K$ relatively open such that $Tz\neq z$ for all $z\in\partial W$ we denote by $i_K(T,W)$ the fixed-point index of $T$ in $W$ with respect to $K$. We are going to prove that there exist $0<r<R$ such that $i_K(T,W_R\setminus\overline{W_r})\neq0$, by showing that $i_K(T,W_r)=1$ while $i_K(T,W_R)=0$. Here $W_s:=B_s(0)\cap K$, where $B_s(0)$ for $s>0$ denotes the ball in $C(\Omegabar)\times C(\Omegabar)$ of radius $s$ and centred at $0$. To prove that there exists $r$ small such that the index in $W_r$ equals $1$, we make use of the superlinearity of $f$ and $g$ at $0$; instead, to prove the index vanishes in $W_R$ for a large $R$, we exploit the a priori bound by Theorem \ref{Thm_Uniform_bound} together with the superlinearity of the nonlinearities at $\infty$.
\vskip0.2truecm
Let us introduce the homotopy $H_1(t,u,v)=tT(u,v)$ and prove that there exists $r>0$ such that $H_1(t,u,v)\neq(u,v)$ for all $t\in[0,1]$ if $(u,v)\in\partial W_r$. 
To this aim, by $(H_7)$ we can find  $\varepsilon\in(0,\lambda_1)$ and $r>0$ such that if $\|u\|_\infty+\|v\|_\infty=r$, one has $f(v)\leq\varepsilon v$ and $g(u)\leq\varepsilon u$. Introduce $\phi_1$ as the first (positive) eigenfunction of $-\Delta$ in $\Omega$. Supposing that there exists $(u,v)$ such that $H_1(t,u,v)=(u,v)$, namely a solution of
\begin{equation*}
	\begin{cases}
		-\Delta u=tf(v)\quad&\mbox{in }\Omega\,,\\
		-\Delta v=tg(u)\quad&\mbox{in }\Omega\,,\\
		u=v=0\quad&\mbox{on }\dOmega\,,
	\end{cases}
\end{equation*}
we test the system with $\phi_1$ and sum the two equations:
\begin{equation*}
	\lambda_1\intOmega(u+v)\phi_1=\intOmega(\nabla u+\nabla v)\nabla \phi_1=t\intOmega\left(f(v)+g(u)\right)\phi_1\leq\varepsilon t\intOmega(u+v)\phi_1\,,
\end{equation*}
which is a contradiction since $t\in[0,1]$ and $\varepsilon<\lambda_1$. As a consequence,
\begin{equation}\label{Index_1}
	i_K(T,W_r)=i_K\left(H_1(1,\cdot,\cdot),W_r\right)=i_K\left(H_1(0,\cdot,\cdot),W_r\right)=i_K(0,W_r)=1\,.
\end{equation}

Define now a second homotopy $H_2(\mu,u,v)=S( f(v+\mu),g(u+\mu))$ where $\mu\ge 0$, the fixed points of which are solutions  of
\begin{equation}\label{sys_mu}
	\begin{cases}
		-\Delta u=f(v+\mu)\quad&\mbox{in }\Omega\,,\\
		-\Delta v=g(u+\mu)\quad&\mbox{in }\Omega\,,\\
		u=v=0\quad&\mbox{on }\dOmega\,.
	\end{cases}
\end{equation}
By $(H_3)$ we know that both $f$ and $g$ are superlinear (see Lemma \ref{TM-subcritiche}), that is, for any $M>0$ there exists $t_0>0$ such that for all $t>t_0$ one has $f(t)>Mt$ and $g(t)>Mt$. Again by testing the system with $\phi_1$, we obtain
\begin{equation*}
	\begin{aligned}
		\intOmega f(v+\mu)\phi_1&=\intOmega\nabla u\nabla\phi_1=\lambda_1\Bigg(\int_{\{u\leq t_0\}}u\phi_1+\int_{\{u>t_0\}}u\phi_1\Bigg)\\
		&\leq\lambda_1\Bigg(t_0\|\phi_1\|_1+\frac1M\intOmega g(u+\mu)\phi_1\Bigg),
	\end{aligned}
\end{equation*}
since $\mu\geq0$ and $\phi_1>0$. Similarly
\begin{equation*}
	\intOmega g(u+\mu)\phi_1=\intOmega\nabla v\nabla\phi_1\leq\lambda_1\Bigg(t_0\|\phi_1\|_1+\frac1M\intOmega f(v+\mu)\phi_1\Bigg).
\end{equation*}
These together yield
\begin{equation*}
	\left(1-\frac{\lambda_1^2}{M^2}\right)\intOmega f(v+\mu)\phi_1\leq\lambda_1\left(1+\frac{\lambda_1}M\right)t_0\|\phi_1\|_1
\end{equation*}
and
\begin{equation*}
	\left(1-\frac{\lambda_1^2}{M^2}\right)\intOmega g(u+\mu)\phi_1\leq\lambda_1\left(1+\frac{\lambda_1}M\right)t_0\|\phi_1\|_1\,.
\end{equation*}
Choosing $M=2\lambda_1$, one then finds
\begin{equation}\label{Bound_l1loc}
	\intOmega f(v+\mu)\phi_1\leq2\lambda_1t_0\|\phi_1\|_1\quad\ \mbox{and}\ \quad\intOmega g(u+\mu)\phi_1\leq2\lambda_1t_0\|\phi_1\|_1\,,
\end{equation}
which by $(H_1)$ and $u,v\geq0$ imply that $\mu$ needs to be bounded from above. Hence, there exists $\overline\mu>0$ such that for all $\mu>\overline\mu$ one has $H_2(\mu,u,v)\neq(u,v)$ in $K$.  Moreover, for any fixed $\mu_0$, we note that $f(\cdot+\mu)$ and $g(\cdot+\mu)$ satisfy the assumptions $(H_1)$-$(H_6')$ uniformly for $\mu\in[0,\mu_0]$. In fact, it is evident that $(H_1)$-$(H_4)$ hold; moreover, $\frac1t\slfrac f{f'}(t+\mu)=\left(1+\frac\mu t\right)\frac1{t+\mu}\slfrac f{f'}(t+\mu)$, which is decreasing for large $t$ and uniformly in $\mu\in[0,\mu_0]$. Indeed, this is a product of two eventually decreasing positive functions by $(H_5)$ applied to $f$; the same argument applies also for $g$, hence $(H_5)$ holds also for $f(\cdot+\mu)$ and $g(\cdot+\mu)$. Similarly, one also verifies $(H_6)$ and $(H_6')$: indeed, 
$$\frac t{(\log t)^2}\frac{f'}f(t+\mu)=\frac{t(\log(t+\mu))^2}{(t+\mu)(\log t)^2}\cdot\frac{(t+\mu)}{(\log(t+\mu))^2}\frac{f'}f(t+\mu)\geq a\,\frac{t(\log(t+\mu))^2}{(t+\mu)(\log t)^2}\geq a^*$$
for $t$ large, for any fixed $a^*\in (1,a)$. The last inequality uses that $\frac{t(\log(t+\mu))^2}{(t+\mu)(\log t)^2}\to1$ as $t\to+\infty$ uniformly in $\mu\in[0,\mu_0]$. Similarly, since $F_\mu(t):=\int_0^t f(s+\mu) \dd s=F(t+\mu)-F(\mu)\leq F(t+\mu)$ for $t$ large one has
\[
t\frac{f'}f(t+\mu)
\geq a\frac{t}{t+\mu}
(\log\log F_\mu(t))^2
\geq a^*(\log\log F_\mu(t))^2
\]
uniformly for $\mu\in[0,\mu_0]$, with $a^*$ as above. Moreover, it is important to note that the constant $\Lambda$ in Proposition \ref{Prop_Lambda} is independent of $\mu$, provided $\mu\in[0,\mu_0]$. Indeed, inspecting its proof in \cite[Theorem 1.2]{dFdOR}, it is sufficient that \eqref{Bound_l1loc} holds with a bound independent of $\mu\in[0,\mu_0]$ to conclude. After all these considerations, we can thus state that Theorem \ref{Thm_Uniform_bound} applies with a bound uniform with respect to $\mu\in[0,\mu_0]$. Therefore taking $\mu_0=\overline\mu+1>0$ there exists $\overline R=\overline R(\overline\mu)>0$ such that all solutions of \eqref{sys_mu} are bounded in $L^\infty$ norm by $\overline R/2$. Hence, $H_2(\mu,u,v)\neq(u,v)$ for all $(u,v)\in B_{\overline R}(0)^c\cap K$ and $\mu\in[0,\overline\mu+1]$. Hence, combining the two information on the map $H_2$, we infer
\begin{equation}\label{Index_0}
	i_K(T,W_{\overline R})=i_K\left(H_2(0,\cdot,\cdot),W_{\overline R}\right)=i_K\left(H_2(\overline\mu+1,\cdot,\cdot),W_{\overline R}\right)=0\,.
\end{equation}
As a result, from \eqref{Index_0} and \eqref{Index_1} one deduces $i_K(T,W_{\overline R}\setminus\overline{W_r})=-1\neq0$, which yields the existence of a nontrivial positive solution of \eqref{sys}.

\section{Open problems}\label{Sec_OP}

Our analysis is a first step in the direction of a comprehensive understanding of uniform bounds for elliptic systems in the conformal setting by means of blow-up techniques. However, the following questions still remain open:
\begin{enumerate}
	\item Is it possible to show that the growth assumption $(H_b)$ is in some sense sharp for the a priori bound? We point out that in the scalar case the Liouville growth $t\mapsto\e^t$ is sharp for \textit{distributional solutions}, in case of a singular potential, see \cite{BM,MR}.
    \vskip0.2truecm
	\item Is it possible to adapt our blow-up method for \eqref{sys} in a generic smooth bounded domain $\Omega\subset\R^2$? One of the main obstructions is the fact that the two components might achieve their blow-up points in different points in $\Omega$ (which, however, cannot approach the boundary by Proposition \ref{Prop_Lambda} from \cite{dFdOR}). Moreover, in order to obtain global estimates for $\tu_k$ and $\tv_k$, we relied on the radial symmetry, see e.g. \eqref{stima_TdD} or \eqref{stima_vkt}.
\end{enumerate}

To end this section, let us spend a few words by commenting on the technical assumptions $(H_6)$ and $(H_6')$. We first prove that both conditions imply bounds from below on the growth of the function $f$. From these, it becomes more evident the fact that $(H_6')$ is a relaxed version of $(H_6)$, since it allows slower growths for $f$.

\begin{lemma}\label{lem_f8_conseq}
	Suppose \eqref{nonBM}, then
	\begin{enumerate}
		\item[\textit{i})] if $(H_6)$ is assumed, then $f(t)\ges\e^{\frac a3(\log t)^3}$ for $t$ large;
		\item[\textit{ii})] if $(H_6')$ is assumed, then $f(t)\ges\e^{\frac a2\log t(\loglog t)^2}\,$ for $t$ large.
	\end{enumerate}
\end{lemma}

\begin{proof}
	\textit{i}) Rewriting $(H_6)$ as
	$$\frac{f'}f(t)\geq a\frac{(\log t)^2}t\,,$$
	and integrating, one infers $\log f(t)\geq\frac a3(\log t)^3+C(t_0)$, from which the estimate follows.
	
	\textit{ii}) One may rewrite $(H_6')$ as
	$$\frac{f'}f(t)\geq\frac at\left(\loglog F(t)\right)^2.$$
	Plainly it is $F(t)>t$ for large $t$, hence
	$$(\log f)'(t)\geq\frac at(\loglog t)^2,$$
	which, integrated, gives
	\begin{equation*}
		\begin{aligned}
			\log f(t)&\geq a\left((\loglog t)^2-2\loglog t+2\right)\log t + C\\
			&=a\log t(\loglog t)^2(1+o(1))\geq\frac a2\log t(\loglog t)^2.
		\end{aligned}
	\end{equation*}
\end{proof}

\begin{remark}
	This growth condition from below is somehow optimal for $(H_6)$-$(H_6')$ to hold. Indeed, e.g. if one considers nonlinearities of the kind $\e^{a\log t(\loglog t)^p}$, then $(H_6')$ is satisfied if $p>2$, or $p=2$ and $a>1$.
\end{remark}

	In light of the analysis of Subcases 1 and 2 in Section \ref{Sec_blowup}, we believe that $(H_6)$-$(H_6')$ may be relaxed to the following, respectively: there exist $t_0>0$ and $a>1$, such that
	\begin{enumerate}
		\item[$(\widetilde{H_6})$] $\displaystyle\ t\frac{f'}{f}(t)\geq a\log t$ for all $t>t_0\,$;
		\item[$(\widetilde {H_6'})$] $\displaystyle\ t\frac{f'}{f}(t)\geq a\log\log F(t)$ for all $t>t_0\,$.
	\end{enumerate}
    These are precisely the assumptions required to complete the proof in Subcase 2, while we suspect that our argument in Subcase 1 might be still slightly improved. Note that, by the same proof of Lemma \ref{lem_f8_conseq}, $(\widetilde{H_6})$-$(\widetilde{H_6'})$ would in fact allow for weaker bounds from below on $f$, that is $f(t)\ges\e^{\frac a2(\log t)^2}$ and $f(t)\ges\e^{\frac a2\log t\log\log t}$ for $t$ large, respectively. Hence, in general we may state the last open question:
    \vskip0.2truecm
    \begin{enumerate}
        \item[3.] To which extent may the assumptions $(H_6)$-$(H_6')$ be relaxed? More interestingly, would it be possible to remove them, in view of the fact that for polynomial nonlinearities the uniform bounds are known to hold, see e.g. \cite{QS}?
    \end{enumerate}
    \vskip1truecm

\section*{Appendix}

Here we collect the proofs of the preliminary lemmas in Section \ref{Sec_consequences_ass}.
\begin{proof}[Proof of Lemma \ref{Lemma3_4}]
	(\textit{i}) By assumption, $\frac{g'}g(t)\to q\in\R^+$ as $t\to+\infty$. Hence, for $t>t_0$ large one has $\frac{g'}g(t)>\delta>0$ for some $\delta\in(0,q)$. Integrating, one has $\log g(t)>\delta t+C_0$, from which the conclusion follows. The estimate for $f$ follows analogously, since $\frac{f'}f(t)\to p\in\R^+$ as $t\to+\infty$.\\
    (\textit{ii}) By de l'H\^opital's theorem and \eqref{nonBM}, 
	$$\frac{\log g(t)}t\to+\infty\quad\ \mbox{and}\quad\ \frac{\log f(t)}t\to0$$
	as $t\to+\infty$. Hence, letting $M>0$, there exists $t_M>0$ such that $\log g(t)>Mt$ and $\log f(t)<\frac tM$ for all $t>t_M$. Choosing $M=1$ we therefore infer $\log f(t)<t<\log g(t)$ for all $t>t_1$, which proves the lemma.
\end{proof}
\begin{proof}[Proof of Lemma \ref{f3_conseq}]
	(\textit{i}) Under \eqref{BM} or \eqref{nonBM}, it is easy to see that
	\begin{equation*}
		\left(\frac{f'}{f}\right)'\!(t)=-\left(\frac f{f'}\right)'\!(t)\left(\frac{f'}f(t)\right)^2\to0
	\end{equation*}
	and similarly
	\begin{equation*}
		\left(\frac{g}{g'}\right)'\!(t)=-\left(\frac{g'}g\right)'\!(t)\left(\frac g{g'}(t)\right)^2\to0\,.
	\end{equation*}
	(\textit{ii})-(\textit{iii}) If \eqref{nonBM} is assumed, by de l'H\^opital's theorem, we compute
	\begin{equation*}
		\lim_{t\to+\infty}\frac{\slfrac{f}{f'}(t)}{t}=\lim_{t\to+\infty}\left(\frac{f}{f'}\right)'(t)=0\,.
	\end{equation*}
	Since $\slfrac{f}{f'}>0$ by $(H_1)$ and $(H_2)$, we get
    \begin{equation}
        \lim_{t\to+\infty}t\,\frac{f'}f(t)=\lim_{t\to+\infty}\frac{t}{\slfrac{f}{f'}(t)}= +\infty\,.
	\end{equation}
	Similarly
	\begin{equation*}
        \lim_{t\to+\infty}\frac1t\frac{g'}g(t)=\lim_{t\to+\infty}\frac{\slfrac{g'}g(t)}{t}=\lim_{t\to+\infty}\left(\frac{g'}g\right)'\!(t)=0\,.
	\end{equation*}
	All other conclusions are trivially true.
    \noindent (\textit{iv}) Under \eqref{BM} the result is trivial. If instead \eqref{nonBM} holds, using Lemmas \ref{Lemma3_4} and \ref{f3_conseq}(\textit{ii}) we get
	\begin{equation*}
		\frac{f'(t)}{f(t)}g(t)=t\frac{f'(t)}{f(t)}\cdot\frac{g(t)}t\to+\infty\,.
	\end{equation*}
\end{proof}
\begin{proof}[Proof of Lemma \ref{TM-subcritiche}]
	By Lemma \ref{f3_conseq}(\textit{iii}), we know that 
	$$
	\lim_{t \to +\infty} \frac{\log f(t)}{t^2} = \lim_{t\to + \infty}\frac{f'(t)}{2 t f(t)}=0\,.
	$$ 
	Therefore
	\begin{equation*}
		\frac{f(t)}{\e^{\gamma t^2}}=\e^{\log(f(t))-\gamma t^2}=\e^{o(t^2)-\gamma t^2 }\to0
	\end{equation*}
	as $t\to+\infty$, for all $\gamma >0$. Similarly, Lemma \ref{f3_conseq}(\textit{ii}) yields 
	$$
	\lim_{t\to +\infty}\frac{\log f(t)}{\log t} = \lim_{t\to +\infty} \frac{t f'(t)}{f(t)} = +\infty\,,
	$$
	so
	$$
	\frac{f(t)}{t^p} =\e^{\log f(t)- p \log t} =\e^{\log f(t)+o(\log f(t))} \to +\infty 
	$$
	as $t \to +\infty$. Since \eqref{BM} and Lemma \ref{f3_conseq}(\textit{iii}) are symmetric for $f$ and $g$, the same proof holds also for $g$.
\end{proof}

\section*{Acknowledgement}
	An important part of this work was carried out during mutual visits of the authors -- at Universidad de Granada in January 2024 and at Universit\`a dell'Insubria, Varese, in January 2025, where L.B. and G.R. were employed, respectively. The authors express their gratitude to the hosting institutions for the hospitality that helped the development of this paper. The authors would like to thank the anonymous referees for their helpful comments and suggestions.

\section*{Funding}
	L.B., G.M., and G.R. are members of the {\em Gruppo Nazionale per l'Analisi Ma\-te\-ma\-ti\-ca, la Probabilit\`a e le loro Applicazioni} (GNAMPA) of the {\em Istituto Nazionale di Alta Matematica} (INdAM). This work was partially supported by the INdAM-GNAMPA Projects 2025 (CUP E5324001950001) titled {\em Regolarit\`a ed esistenza per operatori anisotropi} and \textit{Critical and limiting phenomena in nonlinear elliptic systems} and by the INdAM-GNAMPA Projects 2026 (CUP E53C25002010001) titled \textit{Structural degeneracy and criticality in (sub)elliptic PDEs} and \textit{Strutture Analitiche e Geometriche in PDEs: Regolarit\`a, Fenomeni Critici e Dinamiche Complesse}. L.B. is also partially supported by the ``Maria de Maeztu'' Excellence Unit IMAG, reference CEX2020-001105-M, funded by MCIN/AEI/10.13039/\-501100011033/ and by the Deutsche Forschungsgemeinschaft (DFG, German Research Foundation) - Project-ID 258734477 - SFB 1173.

\end{document}